\documentclass[11pt,twoside, final]{amsart}
\copyrightinfo{0}{Iranian Mathematical Society}
\pagespan{1}{\pageref*{LastPage}}
\usepackage{etoolbox,lastpage}
\commby{}
\date{\scriptsize   Received: , Accepted: .}
\usepackage{amsmath,amsthm,amscd,amsfonts,amssymb,enumerate}
\usepackage{graphicx}		
\usepackage{color}
\usepackage[colorlinks]{hyperref}
 
\newtheorem{theorem}{Theorem}[section]
\newtheorem{proposition}[theorem]{Proposition}
\newtheorem{lemma}[theorem]{Lemma}
\newtheorem{corollary}[theorem]{Corollary}
\theoremstyle{definition}

\theoremstyle{remark}

\numberwithin{equation}{section}
\newtheorem{conjecture}[theorem]{Conjecture}

\newcommand{\R}{\mathbb R}

\newcommand{\Z}{\mathbb Z}
\newcommand{\Q}{\mathbb Q}

\newcommand{\ord}{\operatorname{ord}}
\newcommand{\leg}[2]{\left( \frac{#1}{#2} \right)}  % Legendre symbol

\newcommand{\vE}{\mathcal E} % lattice points on the sphere
\newcommand{\sym}{\operatorname{Sym}}

\newcommand{\dist}{\operatorname{dist}}
\newcommand{\area}{\operatorname{area}}

\newcommand{\ripleyK}{K}

\renewcommand{\^}{\widehat}
\newcommand{\bx}{\mathbf x}
\newcommand{\by}{\mathbf y}

\newcommand{\bm}{\mathbf m}

 \newcommand{\Ccap}{\operatorname{Cap}}
 \newcommand{\sumflat}{\sideset{}{^{\flat}} \sum}
\newcommand{\Zps}{Z}

\begin{document}

%% The title of the paper goes here.  Edit your title.
 
\title[Spatial statistics for lattice points on the sphere   I]{Spatial statistics for lattice points on the sphere I: Individual results}
%% Now edit the following to give First Author name and email:
%% $^*$ for the corresponding author.
 
\author[J.Bourgain]{Jean Bourgain}
%\email{s.a.mousavi@math.uk.ac.ir}

\author[Z. Rudnick]{Ze\'ev Rudnick}
%\author[Z. Rudnick]{Ze\'ev Rudnick$^*$}
%\email{email@some.ir}

\author[P. Sarnak]{Peter Sarnak}
%\email{}
%% If there are three of more authors they are added in the obvious way. 

 % \thanks{$^*$Corresponding author}

  % I omitted the Corresponding author

%------------------------------------------------------------------------------------%
%%
%% Use the following command to make the title for the paper.
%
 %\CoverPage
 \dedicatory{Dedicated to  Freydoon Shahidi on the occasion of his 70-th birthday}

 \maketitle
%
%%% The following environment is needed for the abstract.  
%%%

\begin{abstract}
We study the spatial distribution of point sets on the sphere   obtained from the representation  of a large integer as a sum of three integer squares. We examine several statistics of these point sets, such as the electrostatic potential, Ripley's function, the variance of the number of points in random spherical caps, and the covering radius. Some of the results are conditional on the Generalized Riemann Hypothesis.  
\end{abstract}

\section{Introduction}

The goal of this paper is to study the spatial distribution of point sets on the sphere $S^2$ obtained from the representation  of a large integer as a sum of three squares. 
Some of the results were announced in \cite{BRS Saff}. 

Let $ \vE(n)$ be the set of integer solutions of the equation $x_1^2+x_2^2+x_3^2=n$:
\begin{equation*}
 \vE(n) = \{\bx\in \Z^3: |\bx|^2=n\} \;.
\end{equation*}
This set might be empty; a necessary and sufficient condition for
$\vE(n)\neq \emptyset$ that is for $n$ to be a sum of three squares, is that $n\neq 4^a(8b-1)$.
We denote by 
  \begin{equation*}
  N= N_n:=\#\vE(n) \;.
  \end{equation*}
%The behaviour of $N_n$ is very subtle and it was a fine achievement in the 1930's
%when it was shown that $N_n$ goes to infinity with $n$ (assuming say that $n$ is square-free;
%if $n=4^a$ then there are only six solutions).
It is known that $N_n\ll n^{1/2+o(1)}$ and if there are {\em primitive} lattice points,
that is $\bx=(x_1,x_2,x_3)$ with $\gcd(x_1,x_2,x_3)=1$ (which happens if an only if $n\neq 0,4,7\bmod 8$)
then there is a lower bound of $ N_n\gg n^{1/2-o(1)}$.
%This lower bound is ineffective and indicates that the behaviour of $N_n$ is still far from being understood;
For more details concerning $N_n$ see \S~\ref{sec:arithmetic}.

Once there are many points in $\vE(n)$, one can ask how they distribute on the sphere.  Linnik conjectured, and proved
assuming the Generalized Riemann Hypothesis (GRH), that for $n\neq 0,4,7 \bmod 8$, the projected lattice points  
\begin{equation*}
 \^\vE(n):= \frac 1{\sqrt{n}}\vE(n) \subset S^2
\end{equation*}
%obtained by projecting to the unit sphere, 
become uniformly distributed on the unit sphere $S^2$   as $n\to \infty$ along this sequence.  
That is, for a nice subset $\Omega\subset S^2$ let 
\begin{equation*}
\Zps(n;\Omega):=\#(\^\vE(n) \cap \Omega) .
\end{equation*}
Then as $n\to \infty$ along this sequence 
\begin{equation}\label{unif dist}
\frac 1{N_n}Z(n,\Omega)  \sim \sigma(\Omega)
\end{equation}
where $\sigma$ is the normalized area measure on $S^2$ ($\sigma(S^2)=1$).  
This was proved unconditionally by Duke \cite{Duke, Duke-SP} and Golubeva
and Fomenko \cite{GF}. 

We will consider various statistics of the point sets $\^\vE(n)\subset S^2$, with the aim of comparing these statistics to those of random points, that is $N$ points chosen independently and uniformly, and contrast them with those of ``rigid" point sets, by which we mean points on a planar lattice, such as the honeycomb lattice. 
See \cite{BRS Saff} for a detailed discussion and proofs of the statements below concerning random points.

\subsection{Electrostatic energy}
The electrostatic energy of $N$ points $P_1,\dots,P_N$ on $S^2$ is given by
\begin{equation*}
   E(P_1,\dots, P_N):=  \sum_{i\neq j} \frac 1{|P_i-P_j| } \;.
  \end{equation*}
This energy $E$ depends on both the global distribution of the points
as well as a moderate penalty for putting the points to close to each other.
The configurations with minimal energy are rigid in various senses \cite{Dahlberg} and we will
see below in Corollary~\ref{cor:absolute continuity}
that our points $\^\vE(n)$ are far from being
rigid.

%For random points one has\footnote{Here and elsewhere, $\sim$
%is the usual asymptotic symbol denoting convergence to one of the
%ratio of the two sides.} that $E\sim N^2$.
% but  that $E-N(N-1)$ has %no definite sign.

 %  Our first result is that to leading order the
%point $\^\vE(n)$ have the same energy as the above.
%\begin{theorem}\label{thm:electrostatic energy}
%There is some $\delta>0$ so that
%  \begin{equation}
%   E(\^\vE(n)) =   N^2 +O(N^{2-\delta})
%  \end{equation}
%as $n\to \infty$, $n\neq 0,4,7 \mod 8$.
%\end{theorem}
% We cannot say anything about the sign of $E(\^\vE(n))-N(N-1)$ which according to
%Table~\ref{table energies} appears to vary.

More generally, the Riesz $s$-energy is defined as  
%\marginpar{Discuss the Tamme problem, Thompson, etc} 
\begin{equation*}
   E_s(P_1,\dots, P_N):=  \sum_{i\neq j} \frac 1{|P_i-P_j|^s} \;.
  \end{equation*}

The minimum energy configuration is known to satisfy \cite{Wagner1, Wagner2}
\begin{equation*}
 I(s)N^2-\beta N^{1-\frac s2}\leq \min_{P_1,\dots,P_N} E_s(P_1,\dots,P_N) \leq I(s)N^2-\alpha N^{1-\frac s 2}    
\end{equation*}
when $0<s<2$, for some $0<\alpha \leq \beta<\infty$ (depending on $s$), where 
\begin{equation*}
 I(s)=\iint_{S^2\times S^2} \frac 1{||x - y||^s} d\sigma(x)d\sigma(y) = \frac {2^{1-s}}{2-s} \;.
\end{equation*}

  We will  show that for  $0<s<2$, $\^\vE(n)$ give points with
asymptotically optimal  $s$-energy:
  \begin{theorem}\label{thm:extremal energy}
  Fix $0<s<2$. Suppose $n\to \infty$ such that $n\neq 0,4,7 \mod 8$. Then there is some $\delta>0$ so that
  \begin{equation*}
   E_s(\^\vE(n)) = I(s) N^2 +O(N^{2-\delta}) \;.
  \end{equation*}
  \end{theorem}
For a recent application of this result, see \cite{RWY}.

\subsection{Point pair statistics: Ripley's function}
The point pair statistic and its variants is at the heart of our investigation.
It is a robust statistic as for as testing the randomness hypothesis
and it is called Ripley's function in the statistics literature \cite{SKM}.
For $P_1,\dots, P_N\in S^2$ and $0<r<2$, set
\begin{equation*}\label{def of ripley}
 \ripleyK_r(P_1,\dots,P_N):=\sum_{\substack{i\neq j\\ |P_i-P_j|<r}} 1
\end{equation*}
to be the number of ordered pairs of distinct points at (Euclidean) distance at most $r$ apart.
For fixed $\epsilon>0$, uniformly for $N^{-1+\epsilon}\leq r\leq 2$,
one has that for $N$ random points (the binomial process)
%\footnote{For the Poisson process with intensity $N/\sigma(S^2)$,
%one has $\ripleyK_r = \frac 14 N^2 r^2$}
\begin{equation*}%\label{int.9}
 \ripleyK_r(P_1,\dots,P_N) \sim \frac 14 N(N-1) r^2 \;.
\end{equation*}
%The following upper bound on $\ripleyK_r(\^\vE(n))$ is consistent with
%the random model and is the main result:

Based on the results below as well as some numerical experimentation,
we conjecture that for $n$ square-free the points $\^\vE(n)$ behave
randomly w.r.t. Ripley's statistic at scales
$N_n^{-1+\epsilon}\leq r\leq 2$; that is
\begin{conjecture} \label{conj K}
For squarefree $n\neq 7\bmod 8$, 
\begin{equation*}%\label{neweq1.10}
  \ripleyK_r(\^\vE(n)) \sim \frac{N^2r^2}4, \quad \mbox{as } n\to \infty \;.
\end{equation*}
\end{conjecture} 
 %One of our main results is the following which 
 We show that Conjecture~\ref{conj K} is true at least in
terms of an upper bound which is off only by a multiplicative
constant.
%\marginpar{I only know this for $n$ squarefree} 
\begin{theorem}\label{thm:poisson}
 Assume the Generalized Riemann Hypothesis (GRH). Then for fixed
 $\epsilon>0$ and $N^{-1+\epsilon} \leq r\leq 2$,
\begin{equation*}
 \ripleyK_r(\^\vE(n)) \ll_\epsilon N^2 r^2
\end{equation*} 
for square-free $n\neq 7\bmod 8$, where the implied constant depends only on $\epsilon$.
\end{theorem}
Remark: We do not need the full force of GRH here, but rather that there are no ``Siegel zeros''.

\subsection{Nearest neighbour statistics}
%\marginpar{reword to distinguish from the Saff volume}
Our study of Ripley's function allows us to investigate the distribution of nearest neighbour distances in $\vE(n)$: 
For $N$ points $P_1,\dots, P_N \in S^2$ let $d_j$ denote the distance from $P_j$
to the remaining points. Since the balls about the $P_j$'s of radius $d_j/2$ are disjoint,
it follows from considerations of area that $\sum_{j=1}^N d_j^2 \leq 4^2$.
Hence %by Cauchy's inequality 
the mean value of the $d_j$'s is at most $4/\sqrt{N}$.
For rigid configurations as well as the ones that minimize the
electrostatic energy, each one of the $d_j$'s is of this size
\cite{Dahlberg}. This is not true for random points, however it is
still true for these that almost all the points are of order
$N^{-1/2}$ apart.

%The distribution of nearest neighbour distances $d_j$ is closely connected to Ripley's function $\ripleyK$.
It is more  convenient to work with the squares of these distances.
In order to space these numbers at a scale for which they have a limiting distribution in the random case (see \cite{BRS Saff}),
we rescale them by their mean for the random case, i.e. replace
$d_j^2$ by $\frac N4 d_j^2$.
Thus for $P_1,\dots, P_N\in S^2$ define the nearest neighbour
spacing measure
$\mu(P_1,\dots,P_N)$ on $[0,\infty)$ by
\begin{equation*}
 \mu(P_1,\dots, P_N) := \frac 1N \sum_{j=1}^N \delta_{\frac N4 d_j^2}
\end{equation*}
where $\delta_\xi$ is a delta mass at $\xi\in \R$.
Note that the mean of $\mu$ is at most $1$ and that for random points we have
\begin{equation*}\label{int.17}
 \mu(P_1,\dots,P_N)\to e^{-x}dx,\quad \mbox{as } N\to \infty \;.
\end{equation*}
Based on this and on  numerical experiments
%(see figure~\ref{hist10000001}) 
we conjecture:
\begin{conjecture}
 As $n\to \infty$ along square-free integers, $n\neq 7\bmod 8$,
\begin{equation*}
 \mu(\^\vE(n)) \to e^{-x}dx \;.
\end{equation*}
\end{conjecture}
%\begin{figure}[h]
%\begin{center}
%\includegraphics[width=100mm]{histprime10000001.pdf}
%\caption{A histogram of the scaled minimal spacing between lattice
%points for for $n=179424691$, the $10,000,001$-th prime,  where
%$N_n=94536$, and modulo symmetries there are $1970$ points.
%The smooth curve is the exponential distribution $e^{-s}$.}
%\label{hist10000001}
%\end{center}
%\end{figure}

%One can strengthen Theorem~\ref{thm:poisson} as well as
%\eqref{eq:int upper bd for V} above to give similar sharp upper bounds for short
%intervals (i.e. for points in an annulus rather than a disk) 

Using Theorem~\ref{thm:poisson} and its proof we deduce the following basic result about the nearest neighbour measures $\mu(\^\vE(n))$:
\begin{corollary}\label{cor:absolute continuity}
 Assume GRH. If $\nu$ is a weak limit of the $\mu(\^\vE(n))$, $n\neq 7\bmod 8$ squarefree,  then $\nu$ is absolutely continuous, in fact there is an absolute constant $c_4>0$ such that
\begin{equation*}
 \nu\leq c_4dx \;.
\end{equation*}
\end{corollary}

Corollary~\ref{cor:absolute continuity} implies that the
$\^\vE(n)$'s are not rigid for large $n$ since for rigid
configurations, %(the honeycomb lattice ???),
$\mu_{P_1,\dots,P_N} \to\delta_{\pi/\sqrt{12}}$.
Moreover, since Corollary~\ref{cor:absolute continuity}  implies that  
such a weak limit $\nu$ cannot charge $\{0\}$ positively, it follows that
almost all the points of $\^\vE(n)$ are essentially separated with
balls of radius approximately $N^{-1/2}$ from the rest. 
%Precisely given a sequence $\eta_N$ satisfying $\eta_N=o(N^{-1/2})$, all but
%$o(N)$ of the $N$ points in $\^\vE(n)$ have the ball of radius
%$\eta_N$ about them free of any other points.

 %\marginpar{I removed discussion of the minimum spacing }

\subsection{The number variance in shrinking sets}

We  consider families of sets $\Omega_n$ which shrink as $n\to \infty$, say spherical caps $\Ccap(\xi,r_n)=\{x\in S^2:\dist(x,\xi)\leq r_n\}$ of radius $r_n$, or more generally annuli $A_{r_n,R_n}(\xi) = \{x\in S^2:r_n\leq \dist(x,\xi)\leq R_n\}$. 
Uniform distribution \eqref{unif dist}  remains true if the sets are  allowed to shrink  with $n$ provided $\area(\Omega_n )\gg n^{-\alpha}$ for some small $\alpha>0$,  
%\marginpar{check the best exponent  $\alpha$} 
but one expects this to be true a  as long as the expected number $N_n\cdot \area(\Omega_n)\gg n^\epsilon$. 
%\marginpar{Maybe safer $N_n\cdot \area(\Omega_n)\gg n^\epsilon$?}
%If $N_n\cdot \area(\Omega_n) \to \infty$, then our randomness hypothesis predicts that   \marginpar{Did we discuss it yet?}
%$\Zps(n;\Omega_n)\geq 1$ at least if $\Omega_n$ are``nice" regions. 
This conjecture (stated by Linnik \cite[Chapter XI]{Linnikbook}) has some profound implications. For instance, applied to annuli centered at the north pole,  it implies another conjecture of Linnik, that 
every integer $n$ ($n\neq 0,4,7 \bmod 8$) can be written as a sum of two square and a mini-square: $n=x^2+y^2+z^2$, with $z=O(n^\epsilon)$ for all $\epsilon>0$. 
It also implies an old conjecture   about the gaps between sums of two squares, 
see \S~\ref{sec:Littlewood}.  
%\marginpar{Who to attribute this to? I didn't find mention of Littlewood} 

We can ask for a version for ``random" sets, meaning we fix a nice set $\Omega_n\subset S^2$ and investigate the statistics of the number of points $\Zps(n; g\Omega_n)$ where $g\in {\rm SO}(3)$ is a random rotation. Examples of such sets would be spherical caps $\Ccap(\xi,r_n )$, 
or annuli   $A_{r_n,R_n}(\xi) $, when the center $\xi$ is chosen uniformly on $S^2$ (which is equivalent to choosing a random rotation). 
%In particular this should be so if $\Omega$ is a spherical cap, 

%this being the very strong covering radius conjecture~\ref{conj:covering}. \marginpar{Linnik \cite[chaper XI]{Linnikbook} expects this to hold for spherical caps} 

%\marginpar{harmonize notation $\delta = r/2$}

%This remains true if $\delta$ is allowed to shrink to zero with $n$ provided $\delta\gg %n^{-\alpha}$ for some small $\alpha>0$,  \marginpar{check the best exponent  $\alpha$} but one expects this to be true all  the way down to $\delta\gg n^{-1/4+\epsilon}$, that is as long as the expected number of points in the cap tends to infinity. 

%For a spherical cap $\Ccap(\xi,\delta)$ centered at $\xi\in S^2$, of area $\delta^2$, let  
%$$ Z(n,\delta;\xi):= \#\Ccap(\xi,\delta)\cap \frac 1{\sqrt{n}} \vE(n) $$
%be the number of projected lattice points in the cap. 

%We study  $ Z(n,\delta;\xi)$  for random caps, that is when the center $\xi$ is chosen 
%uniformly on $S^2$. %, and also when we take a further average over $n$.  
The mean value is tautologically equal to the total number of lattice points times the area, that is 
$$ \int_{{\rm SO}(3)} Z(n;g\Omega_n )dg = N_n \sigma(\Omega_n) \;.$$
where $dg$ is the Haar probability measure on ${\rm SO}(3)$.

We turn to study the variance. 
Note that for ``random" points, the variance of the number of points  is the expected number of points, so one expects that 
\begin{conjecture}\label{Conj pois var} 
Let $\Omega_n$ be a sequence of spherical caps, or annuli. 
If $N_n^{-1+\epsilon}\ll \sigma(\Omega_n)\ll N_n^{-\epsilon}$ as  $n\to \infty$,  $n\neq 0,4,7 \bmod 8$, then 
%\marginpar{What is the right range? Maybe safer $N_n \sigma(\Omega_n)\gg n^\epsilon$} 
\begin{equation}\label{eq:conj for Z}
\int_{{\rm SO}(3)} \Big| Z(n;g\Omega_n )  -N_n \sigma(\Omega_n) \Big|^2 dg  \sim N_n  \sigma(\Omega_n).
\end{equation}
\end{conjecture}
%The conjecture is trivially true when $A<n^{-1}$.  (Only for caps??????)  %$\delta<n^{-1/2}$.
%\marginpar{elaborate} 

%We will show that Conjecture~\ref{Conj pois var}  follows from GRH .......

%\subsection{The almost-all covering radius} 

%Let $\chi_\Omega$ denote the characteristic function of $\Omega$, and

%In what follows we assume $N_n\cdot  \area(\Omega)\gg 1$. 

\begin{theorem}\label{annuli prop}
Let $\Omega_n$ be a sequence of spherical caps, or annuli. 
Assume the Lindel\"of Hypothesis for standard $GL(2)/\Q$ L-functions. 
Then for squarefree $n\neq 7\bmod 8$,  we have  
\begin{equation}\label{eqn:new1.3} 
\int_{{\rm SO}(3)} \Big| Z(n;g\Omega_n )  -N_n \sigma(\Omega_n) \Big|^2 dg  
\ll_\epsilon n^\epsilon   N_n    \sigma(\Omega_n) ,\qquad \forall \epsilon>0 \;.
\end{equation}
\end{theorem}

%\marginpar{Explain why annuli are interesting in connection to mini-squares   - we only need caps elsewhere }

\subsection{The covering radius}
Given $P_1,\dots , P_N\in S^2$,
the covering radius $M(P_1,\dots, P_N)$
is the least $r>0$ so that every point of $S^2$ is
within distance at most $r$ of some $P_j$.
An area covering argument shows that for any configuration
$$M(P_1,\dots,P_N) \geq \frac 4{\sqrt{N}}\;.$$
  For random points, $M\leq
N^{-1/2+o(1)}$. 
An effective version of the equidistribution of $\^\vE(n)$
\cite{GF,  Duke-SP}  
yields some $\alpha>0$ such that $M(\^\vE(n)) \ll N_n^{-\alpha}$. 
%\marginpar{Best $\alpha$?} 
Linnik's  conjecture in particular gives 
%\marginpar{This should be attributed to Linnik}
\begin{conjecture}\label{conj:covering}
  $M(\^\vE(n)) = N_n^{-1/2+o(1)}$ as $n\to \infty$.
\end{conjecture}

  We will show (\S~\ref{sec:proof of cor 1.8}) that \eqref{eqn:new1.3} implies a quantitative upper bound on the covering radius towards Conjecture~\ref{conj:covering}:  
\begin{corollary}\label{cor:covering radius}
For     $n\neq 0,4,7\bmod 8$, if  \eqref{eqn:new1.3} holds then   
$$ M(\widehat \vE(n))\ll  N_n^{-1/4-o(1)}\;.
$$
\end{corollary}

 Under the same assumptions, Theorem~\ref{annuli prop} implies that for a sequence of spherical caps $\Ccap(x,r_n)$, of area $A_n$, 
 \begin{equation}
 \sigma\Big\{ x\in S^2: \^\vE(n)\cap \Ccap(x;r_n) =\emptyset \Big\} \ll_\epsilon \frac{n^\epsilon}{N_n  A_n } ,\qquad \forall \epsilon>0\;.
 \end{equation}
 Thus almost all caps  with area $ \gg N_n^{-1+o(1)}$ contain points from $\^\vE(n)$. 
%In particular, the almost all covering exponent is optimally small, being $-1/2$,
%that is almost all points of $S^2$ are at distance at most $N^{-1/2+o(1)}$ from $\^\vE(n)$. 
%
%In the followup paper to this, we establish the basic conjectures  \ref{conj K} and \ref{Conj pois var} for almost all $n$'s. 
Put another way, the almost all covering exponent  
%\marginpar{Check  the definition !} 
\begin{equation}\label{eq:new1.5}
-\sup\Big( \delta: \varlimsup_{n\to \infty} \sigma\Big\{x\in S^2: \^\vE(n)\cap \Ccap(x,N^{-\delta})\neq \emptyset \Big\} = 1\Big)
\end{equation}
 is equal to $-1/2$ (which is optimally small).

The ergodic method developed by Linnik \cite{Linnikbook} that was mentioned in the first paragraph allowed him 
to prove \eqref{unif dist} for $n$'s in special arithmetic progressions, such as those $n$'s for which a fixed auxiliary prime $p$ splits in $\Q(\sqrt{-n})$.  
%\marginpar{Peter: You said you wanted to change something here?}
In \cite{EMV'},  Ellenberg, Michel and Venakatesh outline an argument combining Linnik's method with the spectral gap property for an associated Hecke operator
$T_p$ on $L^2(S^2)$ \cite{LPS1}, to show that for $n$'s restricted to such a sequence, the almost all covering exponent  is equal to $-1/2$ (they carry the details of
the argument for the congruence analogue of the problem in \cite{EMV, EMV'}). 

In the sequel to this paper we examine Conjectures \ref{conj K} and \ref{Conj pois var}  for all $n$'s. 
%\marginpar{Do we really want to be so specific? We can simply just keep the first sentence}
 In particular we establish \eqref{eqn:new1.3} for $n$'s of the form $n=dm^2$ with $d$ fixed and squarefree, 
while if $d$ is varying and $m\gg n^\epsilon$ ($\epsilon>0$ arbitrary) then the almost all covering radius is shown to be $-1/2$. 
The main result in part II will be the proof of Conjectures \ref{conj K} and \ref{Conj pois var}  for almost all $n$. 

\section{Arithmetic background}\label{sec:arithmetic}
\subsection{The number of lattice points $N_n$}
We first recall what is known about the number of lattice points $N_n=\#\vE(n)$,
that is the number of representations of $n$ as a sum of three squares.
Gauss' formula expresses $N_n$ in terms of
class numbers. For $n$ square-free, $n>3$, it says that
\begin{equation*}%\label{Gauss formula}
N_n = \begin{cases} 12 h(d_n),& n=1,2,5,6 \bmod 8\\ 24 h(d_n), & n=3\bmod 8
\end{cases}
%=\frac{24 h(d_n)}{w_n}(1-\leg{d_n}{2})
\end{equation*}
%\begin{verbatim}
% unify r_3(n) and N_n ...
%\end{verbatim}
where  if $n$ is
square-free,  $d_n$ is the discriminant of the imaginary quadratic
field $\Q(\sqrt{-n})$, that is $  d_n =    -4n$  if $ -n=2,3\bmod 4$ and $d_n=-n$ if $ -n=1\bmod 4$, 
and $h(d_n)$ is the class number of $\Q(\sqrt{-n})$.

%where $d_n$, $h(d_n$) and $w_n$ are the discriminant, class number and
%the number of units in the quadratic field $\Q(\sqrt{-n})$.

Using Dirichlet's class number formula, one may then express $N_n$
by means of the special value $L(1,\chi_{-n})$ of the associated quadratic L-function, where 
$\chi_{-n}$ is the corresponding quadratic character
\begin{equation*}
  \chi_{-n}(m) = \leg{d_n}{m}
\end{equation*}
defined in terms of the Kronecker symbol. It is a Dirichlet character
modulo $|d_n|$. The resulting formula, for $n\neq 7\bmod 8$ square-free, is 
\begin{equation}\label{N in terms of L(1,chi)}
 N_n=\frac{24}{\pi}\sqrt{n}L(1,\chi_{d_n}) \;.
\end{equation}
%where $c_n$ only depends on the remainder of $n$ modulo $8$.

For any $n$ we have an upper bound on the
number of such points of
\begin{equation*}
N_n\ll\ n^{1/2+\epsilon}
\end{equation*}
for all $\epsilon>0$.

In order that there be {\em primitive} lattice points
(that is $\bx=(x_1,x_2,x_3)$ with $\gcd(x_1,x_2,x_3)=1$) it is
necessary and sufficient that $n=b^2 m$ with $b$ odd and $m\neq 7\bmod
8$ square-free, equivalently that $n\neq 0, 4,7 \bmod 8$.
If there are primitive lattice points  then by Siegel's theorem we get a
lower bound
\begin{equation*}
N_n\gg n^{1/2-\epsilon} \;.
\end{equation*}

\subsection{The arithmetic function $A(n,t)$}

%\subsection{Definition and basic properties}
Let $A(n,t)$ be the number of (ordered) pairs $(\bx,\by)\in \vE(n)\times \vE(n)$ with inner product $\bx\cdot\by = t$, equivalently
$|\bx-\by|^2=2(n-t)$:
\begin{equation*}
%  \begin{split}
A(n,t)  = \#\{(\bx,\by)\in \Z^3 \times \Z^3: |\bx|^2=|\mathbf
y|^2=n, \;\bx\cdot \by=t \} \\
%&= \#\{\bx,\by\in \Z^3: |u\bx +v\by|^2
%= nu^2 +2tuv +nv^2\}
%  \end{split}
\end{equation*}
which is the number of representions of the binary form $nu^2+2tuv +nv^2$ as a sum of three squares:
\begin{equation}\label{binary quad form}
 \sum_{j=1}^3 (x_ju+y_j v)^2 = nu^2 +2t uv + nv^2 \;.
\end{equation}

The arithmetic function $A(n,t)$ was studied by Venkov \cite{Venkov} \cite[Chaper 4.16]{Venkov book},
Pall \cite{Pall42, Pall 48} and others, who gave an exact formula for it as a product of local densities.
The formulas in  \cite[Theorem 4]{Pall 48} imply that
\begin{equation*}%\label{A(n,t) for n prime}
A(n,t)  = 24\alpha_2(n,t)   \prod_{\substack{ p\mid n^2-t^2\\p\neq 2}} \alpha_p(n,t)
%\sum_{\substack{ d\mid n^2-t^2\\d\mbox{ odd}}}\chi_{-n}(d)
\end{equation*}
the product over odd primes dividing the discriminant $n^2-t^2$, where the factors $\alpha_p(n,t)$ are given as follows:

The $2$-adic density $\alpha_2(n,t)$ equals either one or zero (we will not need to specify when either happens).

To specify $\alpha_p(n,t)$ for odd primes $p$, we need some
notations:
For a prime $p$ and an integer $m$ we denote by $\ord_p(m)$ the
largest integer $k$ so that $p^k\mid m$
(when  $t =0$ we use the convention $\ord_p(0)=\infty$). If $p$ is an
odd prime then $\leg{m}{p}$ is the Legendre symbol. 
%Further, if $n$ is square-free, let $d_n$ be the discriminant of the imaginary quadratic
%field $\Q(\sqrt{-n})$, that is
%\begin{equation*}
%  d_n =
%  \begin{cases}
%    -4n,& -n=2,3\bmod 4\\ -n,& -n=1\bmod 4
%  \end{cases}
%\end{equation*}
%and denote by $\chi_{-n}$ the corresponding quadratic character
%\begin{equation*}
%  \chi_{-n}(m) = \leg{d_n}{m}
%\end{equation*}
%defined in terms of the Kronecker symbol. It is a Dirichlet character
%modulo $|d_n|$.

Assume now that $p$ is odd. Then
the quadratic form \eqref{binary quad form} is equivalent over the
$p$-adic integers $\Z_p$ to a diagonal one
%\begin{equation}
%  T\simeq   \begin{pmatrix} \epsilon_1 p^{a_1}&0 \\0& \epsilon_2 p^{a_2}   \end{pmatrix}
%\end{equation}
\begin{equation*}
 \epsilon_1 p^{a_1} u^2+ \epsilon_2 p^{a_2} v^2
\end{equation*}
with $\epsilon_i$ being $p$-adic units and $0\leq a_1\leq a_2$ are given by
\begin{equation*}
  \begin{split}
a_1 &= \min(\ord_p(n),\ord_p(t))=\ord_p(\gcd(n,t)) \\
% a_1 &= \min(\ord_p(n), \ord_p(t)) , \quad p^{a_1} =
 % \gcd(n,t,p^\infty ) \\
  a_1+a_2 &= \ord_p(n^2-t^2)
  \end{split}
\end{equation*}
%Let $\chi$ be the quadratic residue character $\chi(m) = \leg{m}{p}$.  Then
%According to \cite[Theorem 2]{Kitaoka}, we have
Then
\begin{itemize}
\item If $a_1, a_2= 1 \bmod 2$ then
  \begin{equation}\label{alpha11}
    \alpha_p(n,t) =
p^{\frac{a_1-1}2} \frac{1-\frac 1{p^{(a_1+1)/2}} }{1-\frac 1p}(1+\leg{-\epsilon_1\epsilon_2}{p})
  \end{equation}
%(in particular $\alpha_p=0$ if $-\epsilon_1\epsilon_2$ is not a square in $\Z_p$).
\item If $a_1=1 \bmod 2$, $a_2=0\bmod 2$ then
  \begin{equation}\label{alpha10}
\alpha_p(n,t) =
p^{\frac{a_1-1}2} \frac{1-\frac 1{p^{(a_1+1)/2}} }{1-\frac 1p} (1+\leg{-\epsilon_2}{p})
  \end{equation}
%\begin{verbatim}
%  is \leg{-1}{p}\chi(\epsilon_2) = \chi(-\epsilon_2) ??
%\end{verbatim}
\item If $a_1=0\bmod 2$, $a_2=1\bmod 2$ then

\begin{equation}\label{alpha01}
\alpha_p(n,t)  =  p^{(a_1-2)/2} \frac{1-\frac 1{p^{a_1/2}} }{1-\frac 1p}
(1+ \leg{-\epsilon_1}{p}  )
+ p^{a_1/2} \sum_{k=0}^{a_2-a_1 } \leg{-\epsilon_1}{p}^k
\end{equation}

\item If $ a_1,a_2=0\bmod 2$ then
\begin{equation}\label{alpha00}
  \alpha_p(n,t) = 2p^{(a_1-2)/2} \frac{1-\frac 1{p^{a_1/2}} }{1-\frac 1p}
 + p^{a_1/2} \sum_{k=0}^{a_2-a_1 } \leg{-\epsilon_1}{p}^k
\end{equation}
%In particular, in the generic case $p\nmid 2(n^2-t^2)$ where
%$a_1=a_2=0$, we have $  \alpha_p = 1-\frac 1{p^2}$.
\end{itemize}
%As a consequence we find that
%\begin{equation}
%   \frac {\alpha_p}{1-\frac 1 {p^2}} \leq
%\begin{cases}
% 2 p^{\frac{ a_1-1}2}  ,&   a_1=1\bmod 2 \\
%2p^{a_1/2}(a_2-a_1+1) ,& a_1=0\bmod 2
%\end{cases}
%\end{equation}
In particular, if  $p\nmid 2n$  then $a_1=0$ and $\epsilon_1=n$, so
that $\leg{-\epsilon_1\epsilon_2}{p}=\chi_{-n}(p)$ and
\begin{equation*}
 \alpha_p(n,t) = \sum_{j=0}^{\ord_p(n^2-t^2)} \chi_{-n}(p^j)
\end{equation*}

Moreover, if $n$ is square-free and $p\mid n$ is odd then the above
formulas show that  if $p\nmid t$ (which is equivalent to $p\nmid
n^2-t^2$ in that case) then  $ \alpha_p(n,t) = 1$, while if  $p\mid
\gcd(n,t)$, so that $a_1=1$ and $p^2\mid n^2-t^2$, then $\alpha_p(n,t)\leq 2$.

%\begin{verbatim}
% does this include p|n ?
%      add p|gcd(n,t)
%     \end{verbatim}

%\subsection{Linnik's lemma}\label{sec:Linnik's lemma}

%\begin{verbatim}
% do I need n squarefree (odd ?) here ?
%\end{verbatim}

We use  \eqref{alpha11},  \eqref{alpha10}, \eqref{alpha01}, \eqref{alpha00} to  bound $A(n,t)$ by the value of a
multiplicative function at $n^2-t^2$:
First assume that $n$ is squarefree. 
Let  $f_n$ be the multiplicative function whose values on prime powers are: $f_n(2^k)=1$, while for $p\neq 2$,
\begin{equation}\label{def of fn}
   f_n(p^k) =\begin{cases}
  \sum_{j=0}^k \chi_{-n}(p^j),&  p\nmid  n \\
   1,&   p\mid n \mbox{ and } k=1\\
   2,&  p\mid n \mbox{ and } k\geq 2
   \end{cases}
  \end{equation}
   Then the above computations yield that if $n$ is square-free and  $p$
  is  odd, then  $\alpha_p(n,t) \leq f_n(p^k)$, $k=\ord_p(n^2-t^2)$, hence:  
  \begin{lemma}\label{lem:a(n,t) bd by f}
  If $n$ is square-free, and $|t|<n$, then
   \begin{equation*}
   A(n,t) \leq 24 f_n(n^2-t^2) \;.
  \end{equation*}
  \end{lemma}

More generally, for $n$ which are not square-free, we have
%\marginpar{Might be needed in part II}
\begin{lemma}\label{lem:A(n,t) for nonsquarefree n}
Let 
$$m = \prod_{\ord_p(\gcd(n,t))\geq 2} p^{\ord_p(\gcd(n,t))}$$
and write $n=mn_1,$ $t=mt_1$. Let $f_{m,n}$ be the multiplicative function defined by  
$f_{n,m}(2^k)=1$, and for $p$ odd 
\begin{equation}\label{def of fmn}
f_{m,n}(p^k)=\begin{cases}
k+1,& p\mid m\\ \sum_{j=0}^k \leg{-n}{p}^j,& p\nmid n\\ 
1,& p\nmid m,\; p\mid n,\; k=1 \\ 2,& p\nmid m,\; p\nmid n,\; k\geq 2
\end{cases}
\end{equation}
Then for all $\epsilon>0$, 
\begin{equation*} 
A(n,t) \ll m^{\frac 12}\tau(m) f_{n,m}(n_1^2 - m_1^2)
\end{equation*}
where $\tau(m)$ is the divisor function. 
\end{lemma}
%\marginpar{replaced $m^{\epsilon}$ with the divisor function $\tau(m)$}

\subsection{Linnik's fundamental lemma}\label{sec:Linnik's lemma}

  We will need an upper bound for $A(n,t)$ valid for general $n$:
  \begin{proposition}\label{bd for A(n,t)}
  If $|t|<n$ then
   \begin{equation*}
    A(n,t) \ll \gcd(n,t)^{1/2} n^\epsilon, \quad \forall \epsilon>0\;.
   \end{equation*}
   \end{proposition}
  This kind of bound, a consequence of  Lemma~\ref{lem:A(n,t) for nonsquarefree n}, was stated and used by Linnik \cite{Linnik40},
  who omitted the factor of $\gcd(n,t)^{1/2}$. A correct version was
  given by Pall \cite[\S 7]{Pall42}, \cite[Theorem 4]{Pall 48},
  see also \cite[Section 4]{EMV} for a discussion of the case when $n$ is square-free.

Proposition~\ref{bd for A(n,t)} allows us to deduce a mean
equidistribution statement  for regions which on average contain
{\em  one} lattice point.
  To do so, divide the sphere $\sqrt{n}S^2$  into boxes $\{A_j\}$ of size $\approx n^{1/4}$ (so there are about $n^{1/2}$ such boxes); so one expects that there should be at most $n^\epsilon$ lattice points in each such box.
  We show that this expectation is met in the mean square, that is
  \begin{theorem}\label{thm:mean equidist}
   \begin{equation*}
   \sum_j \Big(\# A_j\cap \vE(n) \Big)^2 \ll n^{1/2+\epsilon}, \quad \forall \epsilon>0
  \end{equation*}
  \end{theorem}

  \begin{proof}
  Theorem~\ref{thm:mean equidist} is an immediate consequence of Proposition~\ref{bd for A(n,t)}, since
  \begin{equation*}%\label{mean square via A(n,t)}
  \begin{split}
   \sum_j \Big(\# A_j\cap \vE(n) \Big)^2 &\ll \#\{x,y\in \vE(n):  |x-y| \ll n^{1/4} \}  \\
  &\ll  \sum_{n-n^{1/2}\leq t\leq n} A(n,t) \;.
  \end{split}
  \end{equation*}
  Applying Proposition~\ref{bd for A(n,t)} now gives
  $$
  \sum_j \Big(\# A_j\cap \vE(n) \Big)^2 \ll n^\epsilon \sum_{n-n^{1/2}\leq t\leq n} \gcd(n,t)^{1/2} \;.
  $$
  Thus it suffices to show that the mean value of $\gcd(n,t)$ over the interval $I = [n-\sqrt{n},n]$ is at most $n^\epsilon$.
  Writing
  $$\gcd(n,t) = \sum_{d\mid n, d\mid t} 1$$
  and switching order of summation gives
  \begin{equation*}
   \begin{split}
    \sum_{t\in I} \gcd(n,t)  &= \sum_{d\mid n} \#\{t\in I: d\mid t \}  
   \leq \sum_{d\mid n} \frac{|I|}{d} +O(1 ) \\
  &\ll |I| \log n +O(n^\epsilon) \ll \sqrt{n}\log n
   \end{split}
  \end{equation*}
  proving Theorem~\ref{thm:mean equidist}.
  \end{proof}

\section{Electrostatic energy}\label{sec:electrostatic}

In this section, we show that $\^\vE(n)$ give points with
asymptotically optimal  $s$-energy:
\begin{equation*}
   E_s(P_1,\dots, P_N):=  \sum_{i\neq j} \frac 1{|P_i-P_j|^s}\;.
  \end{equation*}
  In what follows we take $0<s<2$.

  \begin{theorem}\label{thm:extremal energy}
  Fix $0<s<2$. Suppose $n\to \infty$ such that $n\neq 0,4,7 \mod 8$. Then there is some $\delta>0$ so that
  \begin{equation*}
   E_s(\^\vE(n)) = I(s) N_n^2 +O(N_n^{2-\delta})
  \end{equation*}
 where 
\begin{equation*}
 I(s)=\iint_{S^2\times S^2} \frac 1{||x - y||^s} d\sigma(x)d\sigma(y) = \frac {2^{1-s}}{2-s}\;.
\end{equation*}
  \end{theorem}

  \subsection{A division into close and distant pairs}

  We denote by
\begin{equation*}
 \bx \mapsto \^x=\frac{\bx}{\sqrt{n}}
\end{equation*}
the projection from the sphere $|\bx|^2=n$ to the unit sphere $S^2$.
  We  fix a small $\rho>0$ and divide the pairs of points in
  $\vE(n)\times \vE(n)$ into close
  pairs and distant pairs,
  depending on whether $||\^x-\^y||<n^{-\rho}$ or not. The projected
  points $\^\vE(n)$ are well separated: $||\^x-\^y||\geq n^{-1/2}$,
  hence we may take $\rho \leq 1/2$.
  We treat the contribution of close pairs by using the upper bound of Proposition~\ref{bd for A(n,t)} for
  the number $A(n,t)$ of pairs $\bx,\by\in \vE(n)$ with inner product
  $\langle \bx,\by\rangle =t$, and that of
  the distant pairs by using a quantitative form of the equidistribution
  of the sets $\^\vE(n)$ on the sphere.

  \subsubsection{The contribution of nearby points}
  \begin{lemma}
    The contribution of nearby pairs is bounded by
    \begin{equation*}
      \sum_{\substack{\bx\neq \by\in \vE(n)\\ ||\^x-\^y||<n^{-\rho}}}  \frac
      1{||\^x-\^y ||^s}
  \ll n^{1-\rho(2-s)+\epsilon} \;.
    \end{equation*}
  \end{lemma}
  \begin{proof}
  The squares of the distances between points in $\vE(n)$ are of the
  form $||\^x-\^y||^2 = 2h/n$ for some integer $h$, since
 \begin{equation*}
||\^x-\^y||^2 = \frac{||\bx-\by||^2}n = \frac{2n-2\langle \bx,\by \rangle}n
\end{equation*}
  and $||\bx-\by||^2=2h$ is equivalent to $\langle \bx,\by \rangle = n-h$.
  Hence the number of pairs of points $\bx,\by\in \vE(n)$ at distance
  $||\bx-\by||^2=2h$ is $A(n,n-h)$, that is
\begin{equation*}
 \#\{\bx,\by\in \vE(n): ||\^x-\^y|| = \sqrt{\frac{2h}{n}} \} = A(n,n-h) \;.
\end{equation*}
  Therefore the contribution of close pairs to the sum $E_s$, that is
  pairs of points with $||\^x-\^y||<n^{-\rho}$,  is:
\begin{equation*}
\sum_{\substack{\bx\neq \by\in \vE(n)\\ ||\^x-\^y||<n^{-\rho}}}  \frac 1{||\^x-\^y ||^s} =
n^{s/2}\sum_{1\leq h\leq \frac 12 n^{1-2\rho}} \frac{A(n,n-h)}{(2h)^{s/2}} \;.
\end{equation*}

According to Proposition~\ref{bd for A(n,t)},
  \begin{equation*}
 A(n,n-h) \ll n^\epsilon \gcd(n,n-h)^{1/2} = n^\epsilon \gcd(n,h)^{1/2} \;.
\end{equation*}
  Hence the contribution of close pairs is bounded by
\begin{equation*}
  \sum_{\substack{\bx\neq \by\in \vE(n)\\ ||\^x-\^y||<n^{-\rho}}}  \frac 1{||\^x-\^y ||^s} \ll
  n^{s/2+\epsilon} \sum_{1\leq h\leq   n^{1-2\rho}} \frac{\gcd(n,h)}{(2h)^{s/2}} \ll n^{1-\rho(2-s)+\epsilon}
\end{equation*}
  as claimed.
  \end{proof}

  As a consequence, we may replace the potential $||\^x-\^y||^{-s}$
  by its truncated form
  \begin{equation*}
   F_n(\^x,\^y) = \min( \frac 1{||\^x-\^y||^s}, \frac 1{n^{s\rho}} )
  \end{equation*}
  to get
  \begin{equation}\label{cost of truncation}
   E_s(\vE(n)) =  \sum_{\bx\neq \by\in \vE(n)} F_n(\^x,\^y) +O(n^{1-\rho(2-s)+\epsilon})
  \end{equation}
  where the remainder term is negligible relative to the main term $N_n^2 I(s)$ since $N_n^2 \gg n^{1-\epsilon}$ by Siegel's theorem.

  \subsubsection{Distant pairs}
  For a fixed $\bx_0\in \vE(n)$, consider the $s$-energy sum
  \begin{equation*}
   S(\bx_0):=\frac 1{N_n}  \sum_{\substack{\bx\in \vE(n)\\  \bx\neq \bx_0}} F_n(\^x,\^x_0)
      = \frac 1{N_n} \sum_{\substack{\bx\in \vE(n)\\  \bx\neq \bx_0}}  \min ( n^{s\rho}, \frac 1{||\^x-\^x_0||^s} )
  \end{equation*}
  %here $\rho>0$ satisfies........
  where $N_n=\#\vE(n)$.

  \begin{proposition}\label{prop: use duke}
 For $0<s<2$, there is some $\eta>0$ so that
  as $n\to \infty$, $n\neq 0,4, 7 \mod 8$,
  \begin{equation*}
   S(\bx_0) = I(s) + O(n^{-\eta+s \rho} +n^{-\rho(2-s)})
  \end{equation*}
where
\begin{equation*}
 I(s)=\int_{S^2} \frac 1{||x - \^x_0||^s} d\sigma(x) = \frac {2^{1-s}}{2-s}\;.
\end{equation*}
  \end{proposition}

  As an immediate consequence of Proposition~\ref{prop: use duke} we see, on using $N_n\gg n^{1/2-\epsilon}$, that
  \begin{equation*}
  \sum_{\bx\neq \by\in \vE(n)} F_n(\^x,\^y)   = N_n^2 I(s) + O\left(n^\epsilon (n^{-\eta+s \rho} +n^{-\rho(2-s)})\right)\;.
  \end{equation*}
  Taking into account \eqref{cost of truncation} we get
  \begin{equation*}
   E_s(\vE(n)) = I(s)N_n^2 + O\left(n^\epsilon (n^{-\eta+s \rho} +n^{-\rho(2-s)})\right)\;.
  \end{equation*}
  Taking $\rho = \eta/2$ %(which is quite small)
we find
  \begin{equation*}
   E_s(\vE(n)) = I(s)N_n^2(1+O(n^{-\eta(1-\frac s2)+\epsilon}))
  \end{equation*}
  which proves Theorem~\ref{thm:extremal energy}. It remains to prove Proposition~\ref{prop: use duke}.

  \subsection{Using equidistribution}

  \subsubsection{Discrepancy on $\R/\Z$ }
  We begin with a short review of discrepancy on the circle, see \cite{KN}:
  For a sequence on the circle $X\subset \R/\Z$, we define Weyl sums by
  \begin{equation*}
   W(k,N):=\frac 1N \sum_{n\leq N} e(kx_n) \;.
  \end{equation*}
  Uniform distribution of $X$ is equivalent to $W(k,N)\to 0$ for all
  $k\neq 0$.

  The discrepancy of the sequence is defined as
  \begin{equation*}
   D_N(X):= \sup_I \left | \frac 1N\#\{ n\leq N: x_n \in I \} -
  \mbox{length}(I) \right|
  \end{equation*}
  where the supremum is over all intervals $I\subset \R/\Z$. Uniform
  distribution is equivalent to $D_N\to 0$. A quantitative measure,
  which also allows to treat shrinking intervals, is given by the
  Erd\"os-Tur\'an inequality, one variant being: %(see \cite{Montgomery-ten}):
  For all $M\geq 1$,
\begin{equation*}
    D_N(X) \ll \frac 1{M+1} + \sum_{k=1}^M \frac 1k \left| W(k,N) \right| \;.
\end{equation*}

  We also recall Koksma's inequality on $\R/\Z$, which bounds the sampling error in
  terms of the discrepancy: Let $X\subset [0,1]$ be a sequence of
  points, with discrepancy $D_N(X)$. If $f$ is continuous on $[0,1]$ and of
  bounded variation, with total variation $V(f)$, then
  \begin{equation}\label{Koksma's inequality}
  \left|  \frac 1N \sum_{n\leq N} f(x_n) - \int_0^1 f(x)dx \right| \ll
  D_N(X) \cdot V(f) \;.
  \end{equation}

  \subsubsection{Spherical coordinates}
  Fix a point $x_0$ on the unit sphere $S^2\subset \R^3$, and define spherical coordinates with $x_0$
  as the North Pole as follows: For a point $x\in S^2$, denote by $\theta\in [0,\pi]$ the angle of inclination,
  that is the angle between the zenith direction (the ray between the
  origin and $x_0$) and the ray from the origin to $x$,
  and by $\phi\in [0,2\pi)$ the azimuthal angle, which is the angle
    between a fixed direction in the plane through the origin orthogonal
    to the zenith direction,   and the ray from the origin to the
    projection of $x$ on that plane.
  Thus we have
  \begin{equation*}
   |x-x_0|^2 = 2(1-\cos \theta) \;.
  \end{equation*}
  In these coordinates, the normalized area measure on $S^2$ is $d\sigma = \frac 1{4\pi} \sin \theta d\theta d\phi$.
  % \int_{S^2} Fd \sigma  = \frac 1{4\pi} \int_{\theta=0}^{\pi} \int_{\phi=0}^{2\pi}
  %F(\phi,\theta) \sin \theta d\phi d\theta$$

  We say  that a function on $S^2$ is {\em zonal} if it is invariant under
  rotation around the line between $x_0$ an the origin, that is depends
  only on the angle of inclination $\theta$.
  For any even $2\pi$-periodic function $g(\theta)$ we may define a zonal
  function on the sphere $S^2$ by setting $G(x) = G(\phi,\theta) =
  g(\theta)$.
  The average of $G$  over the sphere is related to the average of $g$
  over the interval $[0,\pi]$ via
  \begin{equation*}\label{rel between means}
    \int_{S^2} G(x) d\sigma(x) = \frac 12 \int_0^\pi g(\theta)\sin
    \theta d\theta \;.
  \end{equation*}

  \subsubsection{Uniform distribution and discrepancy on the sphere}
  Let $\mathcal H_\nu$ be the space of spherical harmonics of degree
  $\nu$. These are eigenfunctions of the Laplace-Beltrami operator on
  $S^2$, with eigenvalue $\nu(\nu+1)$. The dimension of the space is
  $\dim \mathcal H_\nu = 2\nu+1$. The span of all the spherical
  harmonics is dense in $L^2(S^2)$. Hence to prove equidistribution of
  the sets $\mathcal E(n)$  on the sphere it suffices to show that for all
  spherical harmonics $H\in \mathcal H_\nu$ of positive degree, the corresponding Weyl sums
  \begin{equation*}
 W(H,n):= \frac 1{\#\vE(n)} \sum_{\bx\in \vE(n)} H_\nu(\frac{\bx}{\sqrt{n}})
\end{equation*}
tend to zero.

  %For quantitative versions of equidistribution, there is an analogue
  %of the Erd\"os-Tur\'an theorem \cite{Hlawka, Grabner} for the spherical
  %cap discrepancy:
  For a sequence of points $X\subset S^2$, the
  spherical cap discrepancy is defined as
  \begin{equation*}
    D_N(X):=\sup_C \left| \frac 1N \#\{n\leq N: x_n \in C\} - \sigma(C)\right|
  \end{equation*}
  where the supremum is over all spherical caps, and $\sigma$ is the
  normalized area measure.

   A bound for the discrepancy on the sphere, analogous for the Erd\"os-Tur\'an bound, is
  given by  \cite{Grabner}: For all $M\geq 1$,
   \begin{equation}\label{eq:Grabner}
     D_N(X) \ll \frac {1}{M+1} +\sum_{\nu=1}^M \frac 1\nu
     \sum_{j=1}^{\dim \mathcal H_\nu} \left| W(H_{\nu,j},N)\right|
   \end{equation}
  where $H_{\nu,j}$ denotes an orthonormal basis of $\mathcal H_\nu$.

  \subsubsection{Weyl sums on the sphere}
  A fundamental bound for Fourier coefficients of half-integer weight
  forms, due to Iwaniec \cite{Iwaniec}, allows one to prove uniform
  distribution of the points $\vE(n)$ on the sphere
  \cite{Duke,  GF}.
  We will need a quantitative version of that bound given in \cite{GF},
  see also \cite{Duke-SP}: There are constants $\gamma>0$ (small) and $A>0$
%\footnote{I believe we may take $A=2$}
so that if $H_\nu\in \mathcal H_\nu$ is a spherical harmonic of degree $\nu>0$, then
  \begin{equation*}
    W(H_\nu,n) \ll \frac{n^{1/2}}{N_n} n^{-\gamma} \nu^A ||H_\nu||_\infty
  \end{equation*}
  (recall $N_n:=\#\vE(n)$).

  We take $n$'s for which there is a primitive point in $\vE(n)$,
  equivalently $n\neq 0,4,7 \mod 8$, then $\sqrt{n}/N_n\ll
  n^{\epsilon}$, $\forall \epsilon>0$. Moreover, we replace the
  $L^\infty$ norm by the $L^2$ norm via the inequality
  \begin{equation*}
    ||H||_\infty \leq \sqrt{\dim( \mathcal H_\nu)} \cdot ||H_\nu||_2,\quad \forall
    H_\nu\in \mathcal H_\nu
  \end{equation*}
  which gives:
  \begin{lemma}
   There are $\delta>0$, $B>0$ %\footnote{$B=A+\frac 12$}
so that
  \begin{equation*}
     W(H_\nu,n) \ll n^{-\delta} \nu^B ||H_\nu||_2
  \end{equation*}
  for all $H_\nu\in \mathcal H_\nu$, $\nu>0$.
  \end{lemma}

  Applying the discrepancy bound \eqref{eq:Grabner} with $M\simeq n^{\delta/(B+1)}$ we
  get that the spherical cap discrepancy  $D(\^\vE(n))$ satisfies
  \begin{equation*}
    D(\^\vE(n)) \ll n^{-\eta}, \qquad \eta = \frac{\delta}{B+1}\;.
  \end{equation*}

  \subsection{Proof of Proposition~\ref{prop: use duke}}

  Consider the sequence of points in the interval $[0,1]$ given by
  \begin{equation*}
 z_j = \frac{||\^x_j-\^x_0||^2}4 =  \frac{1-\cos \theta_j}2 \in [0,1] \;.
\end{equation*}
  %where $x_0\in \vE(n)$ and $x_j\in \vE(n)$ lie outside of a cap around
  %$x_0$, so that $|\^x_j-\^x_0|\gg n^{-\eta}$.
  %The number $N$ of these $x_j$'s is $\sim r(n)(1+O(n^{-\beta}))$.
  The area (with respect to $\sigma$) of the cap $||\bx-\bx_0||< 2\sqrt{t}$ is $t$, which is the length of
  the interval for the corresponding points $0\leq z=||\bx-\bx_0||^2/4\leq t$. Hence the
  discrepancy of the sequence $z_j$ on the interval $[0,1]$ is bounded
  by the spherical cap discrepancy of the sequence $\^x_j$, which is $\ll
  n^{-\eta}$. Hence by Koksma's inequality \eqref{Koksma's inequality},
for any continuous function $g$ of bounded variation on $[0,1]$ we have
  \begin{equation*}
    \left|  \frac 1N \sum g(z_j) - \int_0^1 g(t)dt \right| \ll
    n^{-\eta}\cdot V(g) \;.
  \end{equation*}

  Now take
  \begin{equation*}
    g_n(z) = \min(   \frac 1{(2z^{1/2})^s},  n^{s\rho} )
  \end{equation*}
  and
  \begin{equation*}
 G_n(\^x) = g_n(\frac{||\^x-\^x_0||^2}4   )=\min(\frac 1{||\^x-\^x_0||^s}, n^{s\rho}) \;.
\end{equation*}
  The total variation of $g_n$ is
  \begin{equation*}
    V(g_n) \ll \max g_n=n^{s\rho} \;.
  \end{equation*}
  Hence we find that
  \begin{equation*}
      \frac 1N_n \sum_{\substack{\bx\in \vE(n)\\ \bx\neq \bx_0}} G_n(\frac{\bx}{\sqrt{n}}) =
  \int_{S^2} G_n(x)  d\sigma(x) + O(n^{-\eta+s\rho}) \;.
  \end{equation*}

  The mean of $G_n$ is
  \begin{equation*}
    \int_{S^2} G_n(x)  d\sigma(x) =
  \int_{S^2} \frac
    1{|| x-\^x_0||^s} d\sigma(x)  + O(n^{-\rho(2- s)})
  \end{equation*}
  since the difference between the two integrals is certainly bounded
  by
  \begin{equation*}
    \int_{|x-\^x_0|<n^{-\rho}} \frac 1{||x-\^x_0||^s} d\sigma(x) =
   \int_0^{n^{-2\rho}/4}  \frac 1{(2\sqrt{z})^s} dz \ll n^{-\rho(2-s)}
\end{equation*}
  (recall we assume that $0<s<2$).

  In conclusion, we find that
  \begin{equation*}
     S(\bx_0)= I(s) + O(n^{-\eta+s\rho} +n^{-\rho(2-s)})
  \end{equation*}
  %$$I(s) = \int_{S^2} \frac 1{||x-x_0||^s}
  %d\sigma(x) = \int_0^1 \frac 1{(2\sqrt{z})^s}dz = \frac 1{2^{s-1}(2-s)}$$
  proving  Proposition~\ref{prop: use duke}. \qed

\section{Upper bounds on  Ripley's function}
\label{sec:poisson}

\subsection{Nair's Theorem} \label{sec:Nair}
%{Mean values of multiplicative functions}
We will need to use a result of M. Nair \cite{Nair} on mean values of
multiplicative functions of polynomial arguments over short
intervals. 
Nair's theorem, following several prior developments in the subject surveyed in \cite{Nair}, deals with the following situation:
Let $\mathcal M$ be the class of multiplicative, non-negative functions $f$  satisfying
\begin{itemize}
 \item $f(p^k)\leq A_0^k$ \ 
\item $f(n)\leq A_1(\epsilon) n^\epsilon$ for all $\epsilon>0$.
\end{itemize}
We are given an integer polynomial $P(t)=\sum_{j=0}^g a_j t^j\in \Z[t]$ of degree $g$, assumed to have distinct roots, with discriminant $D$, and such that $P(t)$ has no fixed prime divisor. We define the height of $P$ by $||P||:=\max_{j} |a_j|$. Let
\begin{equation*}
\rho(m)  =\#\{x\mod m: P(x)=0\mod m\}
\end{equation*}
and let
\begin{equation*}
 \overline{D} = \prod_{\substack{ p^a || D\\ \rho(p)\neq 0}} p^a \;.
\end{equation*}
\begin{theorem}[Nair \cite{Nair}]
Fix $\alpha,\delta\in (0,1)$. Then for $f\in \mathcal M$, $x^\alpha<y<x$, $x\gg ||P||^\delta$,
\begin{equation*}
 \sum_{x-y<m<x} f(|P(m)|) \ll_{\alpha,\delta,A_1} c(\overline{D}) y \prod_{p\leq x}(1-\frac{\rho(p)}{p})\exp (\sum_{p\leq x} \frac{f(p)\rho(p)}p )
\end{equation*}
where the implied constants depend only on the constant $A_1$ for the family $\mathcal M$, on $\alpha, \delta$ and on the reduced discriminant $\overline{D}$. 
%\marginpar{Dependence on $A_0$?}
\end{theorem}

%Unfortunately we cannot just quote the result since we use
%it in a situation where the the statement of the theorem includes an
%unspecified dependence on our asymptotic parameter. We have to
%indicate the places in the proof where the dependence occurs and
%control it in our special case.

We want to use the result for   the multiplicative functions $f_n$ of \eqref{def of fn}, the polynomial $P(t)=n^2-t^2$,  and $x=n-1$.  
%and $y=\lambda \frac{2n}{N_n} =n^{1/2\pm o(1)}$, 
In this  case we have
$$ \rho(m) = \#\{x\mod m: x^2=n^2 \mod m\}
$$
and hence
$$\quad \rho(p) =
\begin{cases} 2& p\nmid 2n \\1 &p\mid 2n    \end{cases}
$$
and moreover $\rho(p^k) = 2$ for $p\nmid 2n$.
In particular
$$ \overline{D} = D = -4n^2\;.
$$
Thus the unspecified dependence on $\overline{D}$ in Nair's theorem is an issue we need to address.

Examining the proof of Nair's theorem shows that there are only two places where the dependence on $\overline{D}$ appears:

a) In \cite[Lemma 2 (iii)]{Nair}, in the estimate
$$ \sum_{m\leq t} \frac{F(m)\rho(m)}m  \ll_{\overline{D}} \exp(\sum_{p\leq t} \frac{F(p)\rho(p)}p)
$$
where $F\in \mathcal M$. The dependence (at the bottom of page 262)
is in bounding the sum over higher prime powers
$$ \sum_{p\leq t}\sum_{\ell \geq 2} \frac{F(p^\ell)\rho(p^\ell)}{p^\ell} \ll 1 .$$
In our case, since $\rho(p^\ell) \leq 2$  this bound is clearly uniform in $\overline{D}\approx n^2$.

b) In the proof of his main theorem, in \cite[equation (6.3) on page 265]{Nair}, he employs the estimate
$$ y\sum_{\substack{ z^{1/2}<a\leq z \\P^+(a) <\log x\log\log x}} \frac{\rho(a)}a \leq c(\overline{D}) y^{7/8}$$
where $z=y^{1/2}$ and $P^+(a)$ denotes the greatest prime factor of $a$.
In our case, use $\rho(a)\ll a^\epsilon$ (independent of $n$) to bound the sum by
$$
y\sum_{\substack{ z^{1/2}<a\leq z \\P^+(a) <\log x\log\log x}} \frac{\rho(a)}a  \ll y \frac{z^\epsilon}{z^{1/2}} \Psi(z;\log x\log\log x)
$$
where $\Psi(x,z)$ is the number of $a<x$ with $P^+(a)<z$, which is known to satisfy (\cite[Lemma 3]{Nair})
$$\Psi(x;\log x \log\log x)\ll \exp(\frac{3\log x}{\sqrt{\log\log x}}) \ll x^\epsilon\;.
$$
Hence in our case we certainly have
$$y \sum_{\substack{ z^{1/2}<a\leq z \\P^+(a) <\log x\log\log x}} \frac{\rho(a)}a \ll y^{7/8}$$
uniformly in $n$ (recall $x^\alpha<y<x$).

%\newpage

\subsection{Reduction to bounding mean values of multiplicative functions}
%In this section we prove  Theorem~\ref{thm:poisson}. 
For $0\leq a<b<n$ we set
\begin{equation}\label{def of M(n;a,b)}
 M(n;a,b) = \#\{|\bx|^2=|\by|^2=n, \; a<|\bx-\by|^2<b \} %= \sum_{n-\frac y2<t<n} A(n,t)
\end{equation}
so that
%\begin{equation}
% \myK_\lambda(\^\vE(n)) = \frac 1{N_n} M(n;0, \frac{4n}{N_n}\lambda)
%\end{equation}
\begin{equation*}
\ripleyK_r(\^\vE(n)) = M(n;0,r^2n) \;.
\end{equation*}

Recall that we denote by $\chi_{-n}$ the quadratic  character
associated to the field $\Q(\sqrt{-n})$.
%which is a Dirichlet character modulo $|d_N|$ where $d_n$ is the
%discriminant of the field.
 We claim
\begin{proposition}\label{prop:M(n,y)}
Fix $0<\alpha<1$. Assume that $n$ is square-free, $n\neq 7\mod 8$, $a<b<n$
and $n^\alpha<b-a<n$. Then
 \begin{equation}\label{bound for M(n;a,b)}
  M(n;a,b) \ll (b-a) \cdot \exp(2\sum_{p<  n} \frac{ \chi_{-n}(p) }{p}  ) \;.
 \end{equation}
\end{proposition}

\begin{proof}
The condition $a<|\bx-\by|^2 < b$ is equivalent to the inner
product of $\bx,\by$ satisfying
\begin{equation*}
 n- \frac b2 <  \bx\cdot \by < n-\frac a2
\end{equation*}
and hence
\begin{equation*}
M(n;a,b) = \sum_{n-\frac b2 <t< n-\frac a2} A(n,t)
\end{equation*}
where $A(n,t)$ is the number of pairs of vectors $\bx, \by\in\vE(n)$
with inner product $\bx\cdot \by =t$, equivalently
$|\bx-\by|^2=2(n-t)$.

According to Lemma~\ref{lem:a(n,t) bd by f}, we may bound $A(n,t)$ by
the value of a multiplicative function at $n^2-t^2$:
\begin{equation*}
 A(n,t) \leq 24 f_n(n^2-t^2)
\end{equation*}
where $f_n$ is the multiplicative function given by \eqref{def of fn}.
%\begin{equation}
% f_n(p^k) = \sum_{j=0}^k \chi_{-n}(p^j), \quad p\nmid 2n
%\end{equation}
%\begin{equation}
% f_n(p) = 1, p\mid n,
%\end{equation}
%\begin{equation}
% f_n(p^k)=2, p\mid n, k\geq 2
%\end{equation}
Therefore we find that we can bound
\begin{equation*}
 M(n;a,b) \ll   \sum_{n-\frac b2 <  t < n-\frac a2} f_n(n^2-t^2) \;.
\end{equation*}
This is a sum of a multiplicative function at polynomial values,
summed over an interval $(n-b/2,n-a/2)$, for which one can give an
upper bound using Nair's theorem \cite{Nair} described in
\S~\ref{sec:Nair}. The conclusion is that
\begin{equation*}
M(n;a,b) \ll  (b-a) \prod_{p< n-\frac a2} (1-\frac{2}{p}) \exp (
\sum_{p<n- \frac a2} \frac{2f_n(p) }p ) \;.
\end{equation*}
Since $f_n(p) = 1+\chi_{-n}(p)$ (for all $p$  with the convention
$\chi_{-n}(p)=0$ if $p\mid 2n$)
%(for $p\nmid 2n$),
%\begin{verbatim}
% what else ?? including p|N with the convention $\chi_{-n}(p)=0$ if $p\mid n$
%\end{verbatim}
we get
\begin{equation*}
M(n;a,b) \ll (b-a) \exp ( 2\sum_{p<n-\frac a2} \frac{\chi_{-n}(p)
}p ) \;.
\end{equation*}
This is \eqref{bound for M(n;a,b)} except that the sum is over
primes $p<n-a/2$ instead of $p<n$. To recover \eqref{bound for
M(n;a,b)}, note that since $0\leq a <n$, we have
\begin{equation*}
\left| \sum_{n-\frac a2<p<n} \frac{\chi_{-n}(p)}{p} \right| \leq
\sum_{n/2<p<n} \frac{1}{p} \ll \frac 1{\log n}
\end{equation*}
by Mertens' theorem, and hence
\begin{equation*}
M(n;a,b) \ll (b-a) \exp ( 2\sum_{p<n } \frac{\chi_{-n}(p) }p )
\end{equation*}
 as claimed.
\end{proof}

\begin{corollary}\label{cor GRH bound}
 Assume that $n$ is square-free,
$n\neq 7\mod 8$. Then assuming the Generalized Riemann Hypothesis
(GRH),
% \begin{equation}\label{GRH bound for M(n;a,b)}
%  \frac 1{N_n } M(n;\alpha \frac{4n}{N_n},\beta \frac{4n}{N_n}) \ll (\beta-\alpha)
% \end{equation}
\begin{equation*} 
\ripleyK_r(\^\vE(n)) \ll N_n^2r^2  \;.
%\Big( \frac{\exp ( \sum_{p<n } \frac{\chi_{-n}(p) }p )}{L(1,\chi_{-n})} \Big)^2
 \end{equation*}
\end{corollary}
%\marginpar{State unconditional version with bound on exceptional set} 
\begin{proof}
%We have 
%\begin{equation}
%\ripleyK_r(\^\vE(n))  = M(n;0,r^2 n)
%\end{equation}

Taking $a=0$ and $b=r^2n$ in Proposition~\ref{prop:M(n,y)} gives, for $n^{-1/2+\delta}<r<n$, that 
\begin{equation*}
\ripleyK_r(\^\vE(n))   = M(n;0,r^2 n) \ll r^2 n \exp ( 2\sum_{p<n } \frac{\chi_{-n}(p) }p ) 
\end{equation*}
Using the Gauss-Dirichlet formula $ N_n=c_n\sqrt{n}L(1,\chi_{d_n})$ of \eqref{N in terms of L(1,chi)} gives for $n$ squarefree
 \begin{equation*}
\ripleyK_r(\^\vE(n))\ll  r^2N_n^2 \Big( \frac{\exp ( \sum_{p<n } \frac{\chi_{-n}(p) }p )}{L(1,\chi_{-n})} \Big)^2
\end{equation*}
It is a consequence of  GRH, that 
\begin{equation*}
 \sum_{p< x} \chi_{-n}(p) \ll x^{1/2}(nx)^\epsilon, \quad \forall \epsilon>0 \;.
\end{equation*}
 This implies
\begin{equation*}\label{lem on GRH}
 \frac 1{L(1,\chi_{-n})} \exp (  \sum_{p<  n} \frac{\chi_{-n}(p) }p ) =O(1)
\end{equation*}
which gives our claim (in fact what we require is the absence of ``Siegel zeros''). 
\end{proof}

%Corollary~\ref{cor GRH bound} immediately implies
%Theorem~\ref{thm:poisson}, since
%\begin{equation}
 %\myK_\lambda(\^\vE(n)) = \frac 1{N_n} M(n;0, \frac{4n}{N_n}\lambda)\ll \lambda
%\end{equation}

We record the corresponding result when $n$ is not necessarily squarefree:
%\marginpar{state unconditionally with a bound on the exceptional set} 
 \begin{corollary}\label{lem:M for nonsf}
 Assume that $n\neq 7\mod 8$ and $ n^{-1/2+\delta}<r<1$. Then assuming the Generalized Riemann Hypothesis,
  \begin{equation*}
 \ripleyK_r(\^\vE(n))   \ll   \Big(\sumflat_{m\mid n} m^{-\frac 12+\epsilon}  \Big) \cdot n L(1,\chi_{-n})^2  r^2 
  \end{equation*}
 for all $\epsilon>0$,  the sum $\sumflat_m$ running over all $m\mid n$ which are squarefull, that is   such that $m=\prod_p p^{k_p}$ with all $k_p\geq 2$. 
 \end{corollary}

 \begin{proof}
  We first show that, as in Proposition~\ref{prop:M(n,y)},  that for $n^{2\delta}<b-a<n$ 
  \begin{equation}\label{bd for M n nonsf}
  M(n;a,b)\ll (b-a)\Big(\sumflat_{m\mid n} m^{-\frac 12+\epsilon}  \Big) L(1,\chi_{-n})^2
  \end{equation}
 which will prove the Lemma, since  $\ripleyK_r(\^\vE(n)) = M(n;0,r^2n)$. 
 
 According to  Lemma~\ref{lem:A(n,t) for nonsquarefree n}
 \begin{equation*}
 A(n,t) \ll m^{1/2}\tau(m) f_{m,n}(n_1^2-t_1^2)
 \end{equation*}
 where 
 $$m = \prod_{\ord_p(\gcd(n,t))\geq 2} p^{\ord_p(\gcd(n,t))}$$
with   $n=mn_1,$ $t=mt_1$, and $f_{m,n}$ is the multiplicative function \eqref{def of fmn}. %defined by  
%$f_{n,m}(2^k)=1$, and for $p$ odd 
%$$f_{m,n}(p^k)=\begin{cases}
%k+1,& p\mid m\\ \sum_{j=0}^k \leg{-n}{p}^j,& p\nmid n\\ 
%1,& p\nmid m,\; p\mid n,\; k=1 \\ 2,& p\nmid m,\; p\nmid n,\; k\geq 2 
%\end{cases} \;.$$
   Therefore, with $a=ma_1$, $b=mb_1$,  
 \begin{equation}\label{intermediate M}
 M(n;a,b) = \sum_{n-\frac b2<t<n-\frac a2} A(n,t)
 \ll \sumflat_m m^{1/2}\tau(m)  \sum_{n_1-\frac{b_1}2<t_1<n_1-\frac{a_1}2} f_{m,n}(n_1^2-t_1^2) 
 \end{equation}
 the sum $\sumflat_m$ running over all $m\mid n$ which are squarefull, that is 
 such that $m=\prod_p p^{k_p}$ with all $k_p\geq 2$. 
 
 For $m$, $n$  fixed estimate the inner sum using Nair's theorem, noting that $b_1-a_1=\frac{ b-a}m \in (n^{2\delta},n)$, obtaining the bound
 \begin{equation*}
 \begin{split}
  \sum_{n_1-\frac{b_1}2<t_1<n_1-\frac{a_1}2} f_{m,n}(n_1^2-t_1^2) &\ll 
  (b_1-a_1)\prod_{p<n_1-\frac{a_1}2} (1-\frac 2p)\exp\Big (\sum_{p<n_1-\frac{a_1}2} \frac{2f_{m,n}(p)}{p} \Big ) 
  \\ &\ll \frac{b-a}m \exp\Big (2\sum_{p\mid m} \frac 1p + 2\sum_{\substack{p<n-\frac{a_1}2\\ p\nmid m}} \frac{\chi_{-n}(p)}{p} \Big)
  \\& \ll  \frac{b-a}m (\log\log m)^C \exp\Big (  2\sum_{ p<n-\frac{a_1}2 } \frac{\chi_{-n}(p)}{p} \Big)
  \\ &\ll \frac{b-a}{m^{1-\epsilon} }L(1,\chi_{-n})^2
  \end{split}
  \end{equation*}
  (the last inequality assumes GRH, or the absence of Siegel zeros). 
  %Hence we obtain 
 % \begin{equation*}
 % M(n;a,b)\ll (b-a)\Big(\sumflat_{m\mid n} m^{-\frac 12+\epsilon}  \Big) L(1,\chi_{-n})^2
  %\end{equation*}
  Inserting into \eqref{intermediate M} proves \eqref{bd for M n nonsf}.  
 \end{proof}

\subsection{Proof of Corollary~\ref{cor:absolute continuity}}
 We now show that weak limits of the nearest neighbour spacing
measures
\begin{equation}\label{def of nearest nbr msr}
 \mu(\^\vE(n)) := \frac 1N \sum_{j=1}^N \delta_{\frac N4 d_j^2}
\end{equation}
are absolutely continuous, in fact that there is some $c_4>0$ so
that any weak limit $\nu$ of \eqref{def of nearest nbr msr} for 
$n\neq 7\bmod 8$ squarefree, 
satisfies $\nu \leq c_4 dx$. For this we need to show that for any
fixed $0\leq \alpha <\beta<\infty$, the proportion of normalized
nearest neighbour spacings $\frac N4 d_j^2$ which lie in the
interval $[\alpha,\beta]$ satisfies
\begin{equation*}
\frac 1{N} \#\{j\leq N: \alpha\leq \frac N4 d_j^2 <\beta \} \leq
c_4(\beta-\alpha) \;.
\end{equation*}
Since the number of normalized nearest neighbour spacings in an
interval is bounded by  the number of all normalized spacings in
that interval, it suffices to show that
\begin{equation*}
\frac 1{N_n} \#\{ |\bx|^2 = |\by|^2 = n: \alpha< \frac {N_n}4
|\frac {\bx}{\sqrt{n}} - \frac{\by}{\sqrt{n}}| <\beta \}\leq
c_4(\beta-\alpha) \;.
\end{equation*}
The LHS is, in the notation of \eqref{def of M(n;a,b)}, equal to 
\begin{equation*}
\frac 1{N_n} M(n; \alpha \frac{4n}{N_n}, \beta \frac{4n}{N_n}) 
\ll \frac{ (\beta-\alpha)n }{N_n^2} 
\exp\Big(2\sum_{p<n} \frac{\chi_{-n}(p)}{p} \Big)
% \leq c_4(\beta-\alpha)
\end{equation*}
by Proposition~\ref{prop:M(n,y)}. Using    \eqref{N in terms of L(1,chi)} for $n\neq 7\bmod 8$ squarefree, we replace $n/N_n^2$ by $ 1/L(1,\chi_{-n})^{2}$, and as in the proof of 
Corollary~\ref{cor GRH bound} we use GRH to deduce that  
$\exp\Big(2\sum_{p<n} \frac{\chi_{-n}(p)}{p} \Big)/L(1,\chi_{-n})^2 = O(1)$. 
\qed

\section{The number  variance}

%\subsection{GRH implies Conjecture~\ref{Conj pois var} (TBC)} TBC

 \subsection{Proof of Theorem~\ref{annuli prop}}
 
 %The variance over all annuli of radii $\rho_1<\rho_2$ is defined to be 
%\begin{equation}
%V(n;\rho_1,\rho_2):=\int_{S^2} \Big( \Zps(n; A_{\rho_1,\rho_2}(z))- A \Big)^2 d\sigma(z)
%\end{equation}
%where $A=\area(A_{\rho_1,\rho_2}(z))$ is the common area of all these annuli. 

%We now prove Theorem~\ref{annuli prop}, in the more general context of annuli rather than caps.  
For $0\leq \rho_1<\rho_2$ and $z\in S^2$ let 
$$  A_{\rho_1,\rho_2}(z)=\{w\in S^2: \rho_1\leq \dist(z,w)\leq \rho_2 \}.
$$
 (so for $\rho_1=0$ we get a spherical cap).  

The variance over all annuli of radii $\rho_1<\rho_2$ is  
\begin{equation*}
V(n;\rho_1,\rho_2):=\int_{S^2} \Big( \Zps(n; A_{\rho_1,\rho_2}(z))- A 
\Big)^2 d\sigma(z)
\end{equation*}
where $A=\area(A_{\rho_1,\rho_2}(z))$ is the common area of all these annuli.

We want to show, that assuming  the Lindel\"of Hypothesis for standard $GL(2)/\Q$ L-functions, for any sequence $n\to \infty$, with $n\neq 7\bmod 8$   squarefree,  we have  
$$ V(n;\rho_1,\rho_2):=\int_{S^2} \Big( \Zps(n; A_{\rho_1,\rho_2}(z))- A 
\Big)^2 d\sigma(z)  \ll_\epsilon n^\epsilon   N_n   \cdot A ,\qquad \forall \epsilon>0 \;,
$$ 
where $A=\area(A_{\rho_1,\rho_2}(z))$ is the common area of the  annuli.

For $m=0,1,\dots$ and $j=1,2,\dots, 2m+1$, let $\phi_{j,m}$ be an orthonormal basis of eigenfunction of the Laplacian $\Delta$ of the sphere, i.e. of the spherical harmonics of degree $m$.  For such a $\phi_{j,m}$ the Weyl sum is defined by 
\begin{equation}
W_{\phi_{j,m}}(n):= \sum_{\bx\in \vE(n)} \phi_{j,m}(\frac{\bx}{|\bx|} )\;.
\end{equation}
Let $k(z,\zeta)$ be a point pair invariant on $S^2$ \cite{Selberg}. Then (in $L^2$)  
%\marginpar{pointwise or just in $L^2$? our $k$ is not smooth}
\begin{equation}\label{ps4}
k(z,\zeta) = \sum_{m=0}^\infty h_k(m) \sum_{j=1}^{2m+1} \phi_{j,m}(z) \phi_{j,m}(\zeta)
\end{equation}
with 
$$
h_k(m) = \int_{S^2} k(z,\zeta) \omega_m(\zeta) d  \zeta 
$$
where $\omega_m(\zeta)$ is the zonal spherical harmonic about $z$, normalized to take value $1$ at $\zeta=z$, 
and $d\zeta=4\pi d\sigma(z)$ is the un-normalized area measure on $S^2$. Thus  (see e.g. \cite{LPS1}) 
\begin{equation*}
h_k(m) = 2\pi \int_0^1 k(t)P_m(t)dt
\end{equation*} 
%\marginpar{Is this with the area $\sigma$ being normalized? }
where $P_m(t)$ is the Legendre polynomial.

We have 
$$ \Zps(n;A_{\rho_1,\rho_2}(z)) = \sum_{\bx\in \vE(n)} k(\frac{\bx}{|\bx|})$$
for the point pair invariant 
\begin{equation*}
k(z,\zeta) = \mathbf 1_{A_{\rho_1,\rho_2}(z)} (\zeta)
\end{equation*}
where $\mathbf 1_\Omega$ is the indicator function of the set $\Omega$, 
and therefore we get from \eqref{ps4} that
\begin{equation}\label{ps7}
V(n;\rho_1,\rho_2) = \sum_{m=1}^\infty h_{\rho_1,\rho_2}(m)^2 \sum_{j=1}^{2m+1} 
|W_{\phi_{j,m}}(n)|^2 \;. 
\end{equation}

The key arithmetic ingredient is the explicit formula for the Weyl sums in terms of special values of L-functions. 
The particular version that we use is due to \cite{BSP} and \cite{BSSP}  
as explicated in \cite{Luo} and coupled with \cite{Kohnen-Zagier}. 
We choose the $\phi_{j,m}$ to be an orthonormal basis of Hecke eigenfunctions for the action of the Hamilton quaternions on $S^2$ (see \cite{Luo}). 
Each such $\phi_{j,m}$ has a Jacquet-Langlands lift to a holomorphic Hecke cusp form $f_{j,m}$ for $\Gamma_0(8)$, of weight $2m+2$. Let $L(s,f)$ and $L(s,\sym^2f)$ denote the finite parts of the corresponding L-functions.  Then for $n$ squarefree 
\begin{equation}\label{sp8}
|W_{\phi_{j,m}}(n)|^2  = c \frac{n^{1/2} L(\frac 12, f_{j,m}) L(\frac 12, f_{j,m}\times \chi_{-n})}{L(1,\sym^2 f_{j,m})} \;.
\end{equation}%\marginpar{Give exact references }
Here $c>0$ is an absolute constant (independent of $\phi_{j,m}$, $m$ and $n$) 
%\marginpar{does $c$ depend on $m$ ?...}
and $\chi_{-n}$ is the quadratic Dirichlet character  corresponding to the extension $\Q(\sqrt{-n})$. For the indefinite ternary form $y^2-xz$, instead of the definite form $x^2+y^2+z^2$ at hand, the explicit formula \eqref{sp8} is given in \cite[(5.1)]{LMY} and it follows in a similar way from \cite{KS} and \cite{BM}. 

From \eqref{sp8} and\footnote{We also need a good lower bound for $L(1,\sym^2 f_{j,m})$, which unconditionally is due to Hoffstein-Lockhart \cite{HL}}  the Lindel\"of Hypothesis applied to the 
L-functions $L(s, f_{j,m})$ and $ L(s, f_{j,m}\times \chi_{-n})$, \eqref{ps7} becomes
\begin{equation}
\begin{split}
V(n;\rho_1,\rho_2) &\ll_\epsilon \sum_{m=1}^\infty h_{\rho_1,\rho_2}(m)^2 n^{1/2} \sum_{j=1}^{2m+1} m^\epsilon n^\epsilon \\
& = n^{\frac 12+\epsilon} \sum_{m=1}^\infty  h_{\rho_1,\rho_2}(m)^2 m^{1+\epsilon},\qquad \forall \epsilon>0 \;.
\end{split}
\end{equation}
The simple estimate  $h_{\rho_1,\rho_2}(m)\ll m^{-3/2}$ (see \cite[page 169]{LPS1}) yields (for any $X\gg 1 $) 
\begin{equation*}
\begin{split}
V(n;\rho_1,\rho_2)&\ll_\epsilon  X^\epsilon n^{1/2+\epsilon} \sum_{m\leq X} m h_{\rho_1,\rho_2}(m)^2   + n^{1/2+\epsilon} \sum_{m>X} m^{-2+\epsilon}\\
& \ll X^{\epsilon} n^{1/2+\epsilon} \int_{S^2} |\chi_{A_{\rho_1,\rho_2} }(\zeta) |^2 d\sigma(\zeta)  + n^{1/2+\epsilon} X^{-1+\epsilon} \;.
\end{split}
\end{equation*}
Choosing $X=n$ gives  %\marginpar{Shouldn't it be $X=A_{\rho_1,\rho_2}^{-1}$?} 
$$
V(n;\rho_1,\rho_2) \ll_\epsilon n^{1/2+\epsilon'} A  \ll   A   N_n^{1+\epsilon''}
$$
as claimed. 

\subsection{Proof of Corollary~\ref{cor:covering radius}}\label{sec:proof of cor 1.8}
 We show that Conjecture~\ref{Conj pois var} implies Corollary~\ref{cor:covering radius}. 
 \begin{proof}
Assume the covering radius of $\widehat \vE(n)$ is bigger than $\rho$, so that  
 there is some point $\xi_0\in S^2$ so that that the cap $\Ccap(\xi_0, \rho)\subset S^2$ contains no projected lattice point $\frac 1{\sqrt{n}}\vE(n)$. Therefore if $0<\delta \leq \rho/2$, then for all $\xi\in \Ccap(\xi_0,\rho/2)$, the caps $\Ccap(\xi,\delta)$ also do not contain any projected lattice points, that is 
$$ 
Z(n;\Ccap(\xi,\delta))= 0, \quad \forall \xi\in \Ccap(\xi_0,\rho)\;. 
$$
It follows that 
\begin{equation}\label{big var Z} 
\int_{S^2} \Big| Z(n;\Ccap(\xi,\delta))- N_n\area(\Ccap(\xi,\delta))  \Big|^2 d\sigma(\xi) \gg \rho^2 N_n^2 \delta^4 \;.
\end{equation}
 
Combining \eqref{big var Z} and Conjecture~\ref{Conj pois var}  gives
$$
\rho^2\delta^2 \ll N_n^{-1}\;.
$$
Taking $\delta = \rho/2$ we obtain
$$\rho\ll N_n^{-1/4} $$
 as claimed.
\end{proof}

\section{Gaps between sums of two squares}\label{sec:Littlewood}

We denote by $\mathcal S_2=\{n_1<n_2<\dots \}$ the sequence of integers which are sums of two squares. An old conjecture asserts that the gaps between consecutive elements of $\mathcal S_2$ satisfy $n_{i+1}-n_i\ll n_i^\epsilon$, for all $\epsilon>0$. 
%It has long been conjectured that every interval of the form $(x,x+x^\epsilon]$ contains a sum of two squares, for all $\epsilon>0$ and  $x\gg_\epsilon 1$.  
%\marginpar{Who to attribute this to? I didn't find anything about Littlewood}
Note that  primes   $p=1\bmod 4$ are also conjectured to have this property, and since such primes are in $\mathcal S_2$ this a fortiori implies the above conjecture. 
%However, all that is known is the elementary fact that intervals $[x,x+L(x)]$ of length $L(x)\gg x^{1/4}$ contain a sum of two squares. 
However, all that is known is the elementary bound $n_{i+1}-n_i\ll n_i^{1/4}$. 
In this section we point out that the covering radius conjecture~\ref{conj:covering} implies the above conjecture on gaps between sums of two squares. 

For $Y\gg 1$, let  $\mathcal S_2(Y)  = \mathcal S_2\cap [Y,2Y)$, and let 
\begin{equation*}
G(Y) = \max\{  n_{i+1}-n_i:  n_i\in \mathcal S_2(Y)\} 
\end{equation*}
be the maximal gap between sums of square in the interval $[Y,2Y)$, 
\begin{equation*}
G(Y) = n''-n'
\end{equation*}
with $n'<n'' $   consecutive elements of $\mathcal S_2(Y)$.
  We want to show that Conjecture~\ref{conj:covering} implies that $G(Y)\ll Y^\epsilon$, for all $\epsilon>0$. 

Assume then that $G(Y)>Y^\epsilon$. 
By Brun's sieve, every interval of length $\geq G(Y)/8 $ contains an integer $m$ which is not divisible by any small prime $p\leq G(Y)^\delta$, for $\delta>0$ sufficiently small. 
Hence we may find an integer $m$ for which
 \begin{equation} \label{eq:brun}
 |m-\frac{n'+n''}4 |<\frac 18 G(Y)
 \end{equation}
 and free of any prime factors less than $G(Y)^\delta$: 
 \begin{equation}\label{sieve}
p\mid m\Rightarrow p>G(Y)^\delta\;.
\end{equation}

Take $n=m^2$  and the point  $\bm:=(0,0,m)\in \vE(n)$ (note $n=1,5\bmod 8$). Then by Conjecture~\ref{conj:covering} there is $\bx=(x_1,x_2,x_3)\in \vE(n)$, $\bx\neq \bm$ so that 
$$
|\bm-\bx|^2 = x_1^2+x_2^2 + (m-x_2)^2< G(Y)^\delta m  \;.
$$
Thus
\begin{equation*}
x_1^2+x_2^2<G(Y)^\delta m %,\quad 0<m-x_3<G(Y)^{\delta/2} \sqrt{m }
\end{equation*}
and since $x_1^2+x_2^2+x_3^2=m^2$, we have
\begin{equation*}
x_1^2+x_2^2 = (m-x_3)(m+x_3)\;.
\end{equation*}
We claim that $m+x_3\in \mathcal S_2$. To see this, note that   if $p=3\bmod 4$ divides the sum of  two squares $x_1^2+x_2^2$, then $\ord_p(x_1^2+x_2^2)$ is even.  
It follows that if $p=3\bmod 4$ is a prime such that $p\mid m+x_3$ and $\ord_p(m+x_3)$ is odd, then $p\mid m-x_3$ and hence $p\mid m$. 
Since moreover 
\begin{equation*}
m-x_3=\frac{x_1^2+x_2^2}{m+x_3}<\frac{x_1^2+x_2^2}{m}<G(Y)^\delta
\end{equation*}
we conclude that $p\leq m-x_3<G(Y)^\delta$, which is excluded by \eqref{sieve}. Hence $\ord_p(m+x_3)$ is even for any prime $p=3\bmod 4$, that is $m+x_3\in \mathcal S_2$ is a sum of two squares.  

Since $2m=(m+x_3)+(m-x_3)$, we obtain 
\begin{equation*}
\dist(2m,\mathcal S_2)<G(Y)^\delta\;.
\end{equation*}
 Hence 
 \begin{equation*}
 \frac 12 G(Y) =\dist(\frac{n'+n''}{2},\mathcal S_2) \leq |\frac{n'+n''}{2}-2m| + \dist(2m, \mathcal S_2)
 <\frac 14 G(Y) + G(Y)^\delta
 \end{equation*}
 by \eqref{eq:brun}. This is a contradiction for $Y\gg 1$.

\section*{\bf Acknowledgments}
We thank G. Harcos, M. Raziwi\l\l  \; and A. Venkatesh for their insightful comments.

J.B. was partially supported by NSF grant  DMS-1301619. 
Z.R.   was supported by the Friends of the Institute for Advanced Study, 
and by the European Research Council under the European Union's Seventh
Framework Programme (FP7/2007-2013)/ERC grant agreement
n$^{\text{o}}$ 320755.
P.S.  is partially supported by NSF grant DMS-1302952.
 %The United States Government is authorized to reproduce and distribute reprints notwithstanding any copyright notation herein.

%\begin{thebibliography}{20}


\begin{thebibliography}{99}
%\bibitem{Atiyah-Sutcliffe}
%M. Atiyah and P. Sutcliffe, Polyhedra in physics, chemistry and geometry.  Milan J. Math.  71  (2003), 33--58.



%\bibitem{Barban}
%M.B. Barban, The "Large Sieve" method and its applications in the
%theory of numbers, Russian Math. Surveys {\bf 21}, 49--103 (1966).


\bibitem{BM}
E. Baruch and Z. Mao. {\em Central value of automorphic L-functions}. Geom. Funct. Anal. 17 (2007), no. 2, 333--384.

\bibitem{BSP}
S. B\"ocherer and R.  Schulze-Pillot. 
{\em The Dirichlet series of Koecher and Maass and modular forms of weight $3/2$}. 
Math. Z. 209 (1992), no. 2, 273-287.

\bibitem{BSSP}
S. B\"ocherer, P. Sarnak and R.  Schulze-Pillot. 
{\em Arithmetic and equidistribution of measures on the sphere}.  
Comm. Math. Phys. 242 (2003), no. 1-2, 67--80.
 
%\bibitem{Boxer}
%G. Boxer, {\em Small scale distribution of Heegner points}, Princeton
%University undergraduate thesis, 2010.

%\bibitem{BRrestriction3}
%J. Bourgain and Z. Rudnick,  {\em  Restriction of toral
%eigenfunctions to hypersurfaces and nodal sets},  Geometric and Functional Analysis: Volume 22, Issue 4 (2012), Page 878--937.

\bibitem{BRS Saff}
J.  Bourgain, Z.  Rudnick and P. Sarnak. {\em Local statistics of lattice points on the sphere}.  Contemporary Mathematics vol. 616, (2016), 269--282 (proceedings of Constructive Functions 2014).


%\bibitem{Bykovskii}
%V.A. Bykovski\u{i},
%{\em A trace formula for the scalar product of Hecke series and its applications}.
%Zap. Nauchn. Sem. S.-Peterburg. Otdel. Mat. Inst. Steklov. (POMI) 226 (1996), Anal. Teor. Chisel i Teor. Funktsii. 13, 14--36, 235--236; translation in J. Math. Sci. (New York) 89 (1998), no. 1, 915--932

%\bibitem{Cassels}
%J. W. S. Cassels, {\em Rational quadratic forms}. Academic Press, London, 1978.

%\bibitem{Conway-Sloane}
% J.~H.~Conway and N.~J.~A.~Sloane, {\em Sphere packings, lattices and groups}.
% Third edition.
 %With additional contributions by E. Bannai, R. E. Borcherds, J. Leech, S. P. Norton, A. M. Odlyzko,
 %R. A. Parker, L. Queen and B. B. Venkov.
% Grundlehren der Mathematischen Wissenschaften, 290.
 %[Fundamental Principles of Mathematical Sciences]
 % Springer-Verlag, New York, 1999.

\bibitem{Dahlberg}
B. Dahlberg, {\em On the distribution of Fekete points}.
Duke Math. J.  {\bf 45}  (1978), no. 3, 537--542.

\bibitem{Duke}
W. Duke,
{\em Hyperbolic distribution problems and half-integral weight Maass forms}.
Invent. Math. 92 (1988), no. 1, 73--90.

\bibitem{Duke-SP}
W. Duke and R. Schulze-Pillot, {\em Representation of integers by positive ternary quadratic forms and equidistribution of lattice points on ellipsoids}.  Invent. Math.  99  (1990),  no. 1, 49--57.

\bibitem{EMV}
J. Ellenberg, P. Michel and A. Venkatesh, {\em Linnik's ergodic
method and the distribution of integer points on spheres}.
Automorphic representations and L-functions,  119--185, Tata Inst.
Fundam. Res. Stud. Math., 22, Tata Inst. Fund. Res., Mumbai, 2013.

\bibitem{EMV'}
J. Ellenberg, P. Michel and A. Venkatesh, {\em Linnik's ergodic
method and the distribution of integer points on spheres},  preprint (2010), arXiv:1001.0897 [math.NT].

\bibitem{GF}
Golubeva, E. P.; Fomenko, O. M.
Asymptotic distribution of lattice points on the three-dimensional sphere.   Zap. Nauchn. Sem. Leningrad. Otdel. Mat. Inst. Steklov. (LOMI) 160 (1987), Anal. Teor. Chisel i Teor. Funktsii. 8, 54--71, 297; translation in
J. Soviet Math. 52 (1990), no. 3, 3036--3048

\bibitem{Grabner}
Grabner, Peter J. Erd\"os-Tur\'an type discrepancy bounds.  Monatsh. Math.  111  (1991),  no. 2, 127--135.

%\bibitem{Harman}
%G. Harman, {\em Approximation of real matrices by integral matrices.}
%J. Number Theory  34  (1990),  no. 1, 63--81.

%\bibitem{Hironaka-Sato}
%F. Sato, and Y. Hironaka, {\em Local densities of representations of quadratic forms over $p$-adic integers (the non-dyadic case)}.
% J. Number Theory  83  (2000),  no. 1, 106--136.


\bibitem{HL}
J. Hoffstein and P. Lockhart, 
{\em Coefficients of Maass forms and the Siegel zero}. 
With an appendix by Dorian Goldfeld, Hoffstein and Daniel Lieman. 
Ann. of Math. (2) 140 (1994), no. 1, 161--181.

%\bibitem{peres et al}
%J.B. Hough, M. Krishnapur, Y. Peres    and B. Vir\'ag,
%{\em Zeros of Gaussian analytic functions and determinantal point processes}.
%University Lecture Series, 51. American Mathematical Society,
%Providence, RI, 2009.

\bibitem{Iwaniec}
H. Iwaniec,
Fourier coefficients of modular forms of half-integral weight.
Invent. Math. 87 (1987), no. 2, 385--401.


\bibitem{KS}
S. Katok and P. Sarnak. 
{\em Heegner points, cycles and Maass forms}. 
Israel J. Math. 84 (1993), no. 1-2, 193--227. 

%\bibitem{Kingman}
%J. F. C. Kingman,  Poisson processes. Oxford Studies in Probability, 3. Oxford Science Publications. The Clarendon Press, Oxford University Press, New York, 1993.

%\bibitem{Kitaoka}
%Y. Kitaoka, {\em A note on local densities of quadratic forms}. Nagoya Math. J. 92 (1983), 145--152.


%\bibitem{Kitaoka book}
%Y. Kitaoka, {\em Arithmetic of quadratic forms}.
%Cambridge Tracts in Mathematics, 106. Cambridge University Press, Cambridge, 1993.

\bibitem{Kohnen-Zagier}
 W. Kohnen and D.  Zagier. {\em Values of L-series of modular forms at the center of the critical strip}. Invent. Math. 64 (1981), no. 2, 175--198. 
 
\bibitem{KN}
Kuipers, L.; Niederreiter, H. Uniform distribution of sequences. Pure
and Applied Mathematics. Wiley-Interscience [John Wiley \& Sons],
%New York-London-Sydney, 1974.

%\bibitem{Lavrik}
%A. F. Lavrik, On the moments of the class number of primitive
%quadratic forms with negative discriminant,
%Dokl. Akad. Nauk SSSR {\bf 197} (1971).

\bibitem{Linnik40}
U. Linnik, {\em \"Uber die Darstellung grosser Zahlen durch positive
tern\"are quadratische Formen}.  Bull. Acad. Sci. URSS. Ser. Math.
[Izvestia Akad. Nauk SSSR] {\bf 4} (1940), 363--402.

\bibitem{Linnikbook} 
Yu. V. Linnik. Ergodic properties of algebraic fields. Translated from the Russian by M. S. Keane. 
Ergebnisse der Mathematik und ihrer Grenzgebiete, Band 45 Springer-Verlag New York Inc., New York 1968. 


\bibitem{LMY}
S.-C. Liu, R. Masri and M.P.  Young.  
{\em Subconvexity and equidistribution of Heegner points in the level aspect}. 
Compos. Math. 149 (2013), no. 7, 1150--1174. 

\bibitem{LPS1}
A. Lubotzky, R. Phillips and P.  Sarnak. {\em Hecke operators and distributing points on the sphere.} I. 
Frontiers of the mathematical sciences: 1985 (New York, 1985). Comm. Pure Appl. Math. 39 (1986), no. S, suppl., S149--S186.
 
\bibitem{Luo}
W.  Luo, {\em A note on the distribution of integer points on spheres}. Math. Z. 267 (2011), no. 3--4, 965--970. 

\bibitem{Nair}
M. Nair, {\em Multiplicative functions of polynomial values in short intervals}.
Acta Arith.  {\bf 62}  (1992),  no. 3, 257--269.

\bibitem{Pall42}
G. Pall,
{\em Quaternions and sums of three squares}.
Amer. J. Math. {\bf 64}, (1942). 503--513.

\bibitem{Pall 48}
G. Pall, {\em Representation by quadratic forms}.  Canadian J. Math.  {\bf 1},  (1949). 344--364.

%\bibitem{Ripley 1976}
%B.D. Ripley, {\em The second-order analysis of stationary point processes}.
%J. Appl. Probability  {\bf 13}  (1976), no. 2, 255--266.

%\bibitem{Ripley 1977}
%B.D. Ripley, {\em Modelling spatial patterns}. With discussion.
%J. Roy. Statist. Soc. Ser. B  {\bf 39}  (1977), no. 2, 172--212.


%\bibitem{Saff-K}
% Saff, E. B.; Kuijlaars, A. B. J. Distributing many points on a sphere. Math. Intelligencer 19 (1997), no. 1, 5--11.

%\bibitem{Siegel annals 1935}
%C.L. Siegel, {\em \"Uber die analytische Theorie der quadratischen Formen}.  Ann. of Math. (2) 36 (1935), no. 3, 527--606.

%\bibitem{Siegel book}
%Siegel, Carl Ludwig {  Lectures on the analytical theory of quadratic forms}. Notes by Morgan Ward.
%Third revised edition. Buchhandlung Robert Peppm\"uller, G\"ottingen 1963


\bibitem{RWY}
Z. Rudnick, I. Wigman and N. Yesha. {\em  Nodal intersections for random waves on the 3-dimensional torus}.  Annales de l'institut Fourier, to appear.


\bibitem{Selberg}
A. Selberg. {\em Harmonic Analysis and Discontinuous Groups in Weakly
Symmetric Riemannian Spaces with Applications to Dirichlet Series}, J.
Indian Math.Soc. B. 20 (1956), pp.47--87.

\bibitem{SKM}
D. Stoyan, W.S. Kendall and J. Mecke. {\em Stochastic geometry and its applications}. %With a foreword by D. G. Kendall.
Wiley Series in Probability and Mathematical Statistics: Applied Probability and Statistics. John Wiley \& Sons, Ltd., Chichester, 1987.

%\bibitem{Thompson}
%J.J. Thomson, Philos. Mag. 7 (1904), 237; Philos. Mag. 41 (1921), 510


\bibitem{Venkov}
B.A Venkov (Wenkov), {\em \"Uber die Klassenzahl positiver bin\"arer
quadratischer Formen}.
 Math. Zeitschr {\bf 33} (1931), 350--374.

\bibitem{Venkov book}
 B.A Venkov,  Elementary number theory.  Translated from the Russian and edited by Helen Alderson.
Groningen : Wolters-Noordhoff, 1970.

%\bibitem{Vinogradov}
% I.~M.~Vinogradov, {\em Sur la distribution des r\'esidus et des non-r\'esidus
%des puissances} . Fiz.-Mat. Ob . Permsk. Gos. Univ. 1 (1918), pp. 94--98.

\bibitem{Wagner1}
G. Wagner, On the means of distances on the surface of a sphere (lower bounds), Pacific J. Math. 144 (1990), 389-398.
\bibitem{Wagner2}
G. Wagner, On means of distances on the surface of a sphere. II. Upper bounds.
Pacific J. Math. 154 (1992), no. 2, 381--396.

%\bibitem{Whyte}
%L.L. Whyte,  Unique arrangements of points on a sphere.  Amer. Math. Monthly  59,  (1952). 606--611.

%\bibitem{Yang}
%T. Yang, {\em explicit formula for local densities of quadratic forms}.  J. Number Theory  72  (1998),  no. 2, 309--356.

%\bibitem{Yang2}
%T. Yang, {\em Local densities of 2-adic quadratic forms}. J. Number Theory 108 (2004), no. 2, 287--345.

\end{thebibliography}
\end{document}